\documentclass{amsart}
\usepackage{amssymb, amsthm, amsmath}
\usepackage{graphicx}
\usepackage{enumerate}
\usepackage{ragged2e}
\usepackage{stmaryrd}
\usepackage{mathrsfs}
\usepackage{hyperref}
\usepackage[all]{xy}

\numberwithin{equation}{section}
\newtheorem{thm}[equation]{Theorem}
\newtheorem{lem}[equation]{Lemma}
\newtheorem{prop}[equation]{Proposition}
\newtheorem{cor}[equation]{Corollary}

\newtheorem{question}[equation]{Question}
\theoremstyle{definition}

\theoremstyle{remark}
\newtheorem{rem}[equation]{Remark}

\newtheorem{example}[equation]{Example}

\newcommand{\Hom}{\operatorname{Hom}}
\newcommand{\Ext}{\operatorname{Ext}}
\newcommand{\Tor}{\operatorname{Tor}}
\newcommand{\FP}{\operatorname{FP}}
\newcommand{\HNN}{\operatorname{HNN}}

\begin{document}

\title{Bieri-Eckmann Criteria for Profinite Groups}
\author{Ged Corob Cook}
\address{Department of Mathematics\\
Royal Holloway, University of London\\
Egham\\
Surrey TW20 0EX\\
UK}
\email{Ged.CorobCook.2012@live.rhul.ac.uk}
\thanks{This work was done as part of the author's PhD at Royal Holloway, University of London, supervised by Brita Nucinkis.}

\begin{abstract}

In this paper we derive necessary and sufficient homological and cohomological conditions for profinite groups and modules to be of type $\FP_n$ over a profinite ring $R$, analogous to the Bieri-Eckmann criteria for abstract groups. We use these to prove that the class of groups of type $\FP_n$ is closed under extensions, quotients by subgroups of type $\FP_n$, proper amalgamated free products and proper $\HNN$-extensions, for each $n$. We show, as a consequence of this, that elementary amenable profinite groups of finite rank are of type $\FP_\infty$ over all profinite $R$. For any class $\mathcal{C}$ of finite groups closed under subgroups, quotients and extensions, we also construct pro-$\mathcal{C}$ groups of type $\FP_n$ but not of type $\FP_{n+1}$ over $\mathbb{Z}_{\hat{\mathcal{C}}}$ for each $n$. Finally, we show that the natural analogue of the usual condition measuring when pro-$p$ groups are of type $\FP_n$ fails for general profinite groups, answering in the negative the profinite analogue of a question of Kropholler.

\end{abstract}

\maketitle

\section*{Introduction}

An abstract group $G$ is of type $\FP_n$ over an abstract ring $R$ if the $R[G]$-module $R$ with trivial $G$-action has a projective resolution which is finitely generated for the first $n$ steps; Bieri (\cite[Theorem 1.1.3]{Bieri}) gives necessary and sufficient homological and cohomological conditions, the Bieri-Eckmann criteria, for an $R$-module to be of type $\FP_n$, and hence conditions for $G$ to be of type $\FP_n$ over $R$. 

Analogously, for $G$ and $R$ profinite, there is a profinite group ring $R \llbracket G \rrbracket $, and $G$ is of type $\FP_n$ if $R$ has a projective resolution by $R \llbracket G \rrbracket $-modules which is finitely generated for the first $n$ steps. Some results are known when $G$ and $R$ are pro-$p$: Symonds and Weigel \cite[Proposition 4.2.3]{S-W} give a necessary and sufficient condition for (virtually) pro-$p$ groups to be of type $\FP_n$. However, this has never been studied before for profinite groups. We show in this paper that whether a profinite group is of type $\FP_n$ is measured by whether its homology or cohomology groups commute with direct limits, in a certain sense. Explicitly, we consider certain direct systems of profinite modules whose direct limits as topological modules have underlying abstract modules that are isomorphic, but may not have the same topology. Applying homology or cohomology to these and then comparing the direct limits which result gives criteria for the $\FP_n$ type of $G$, a trick which allows us to forget the topologies of our systems without losing too much information. In this sense our Theorem \ref{b-ehom} is the analogue of \cite[Theorem 1.1.3]{Bieri}. We call these equivalent conditions Bieri-Eckmann criteria, by analogy with the abstract case.

We describe the structure of this paper. Section \ref{a-pmod} describes the mathematical objects, categories and functors we will be using in the rest of the paper, along with some basic results on the relations between them. The natural coefficient category for profinite group cohomology is that of discrete torsion modules, but this turns out not to contain enough information. Instead, we use profinite modules as coefficients. These do not have enough injectives, but we can define our $\Ext$ functors using projectives in the first variable.

In Section \ref{dspm}, we prove all the main technical results. For abstract rings and modules, the $\Ext$ functors commute with direct limits in the second variable, and the $\Tor$ functors commute with direct products in the second variable, if and only if the first variable is of type $\FP_\infty$. The idea here is to show that for profinite modules the $\Ext$ functors (with profinite coefficients) and the $\Tor$ functors commute with direct limits in the second variable if and only if the first variable is of type $\FP_\infty$; the obstruction is that direct limits of profinite modules in the category of topological modules are not necessarily profinite. Thus we take derived functors to functor categories over a small category $I$ which corresponds to a directed poset, and take direct limits afterwards, as described above. This allows us to obtain the Bieri-Eckmann criteria for profinite modules.

In Section \ref{groupdspm} we collect known results about the $\FP_n$ type of profinite groups, and apply the conclusions of Section \ref{dspm} to prove new results, giving Bieri-Eckmann criteria for groups and allowing us to build new groups of type $\FP_n$ from old ones, as promised in the abstract. Then in Section \ref{appl} we define the class of elementary amenable profinite groups (which contains the soluble groups) and use the results of the previous section to show that elementary amenable profinite groups of finite rank are of type $\FP_\infty$ over all profinite $R$; we also construct groups of type $\FP_n$ but not $\FP_{n+1}$ over certain completions of $\mathbb{Z}$, mirroring the construction in \cite[Proposition 2.14]{Bieri}.

Finally, Section \ref{altfin} explores an alternative finiteness condition $\FP'_n$, and shows that, though it is equivalent to being of type $\FP_n$ for pro-$p$ groups, it is not equivalent for profinite groups. The conditions $\FP_n$ and $\FP'_n$ are two different ways to generalise the pro-$p$ condition \cite[Proposition 4.2.3]{S-W} mentioned above. As promised in the abstract, we consider \cite[Open Question 6.12.1]{R-Z}, a question about homological finiteness conditions for pro-$p$ groups, which correspondingly has two possible ways of generalising to profinite groups: we show that the answer to one of these two questions is no.

\section{Abstract and Profinite Modules}
\label{a-pmod}

Let $R$ be a commutative profinite ring, and $\Lambda$ a profinite $R$-algebra. We define the categories $PMod(\Lambda)$, $DMod(\Lambda)$ and $TMod(\Lambda)$ to be the categories of profinite, discrete and topological left $\Lambda$-modules, respectively, with continuous $\Lambda$-module homomorphisms as their morphisms. We require for all of these that the $\Lambda$-action be continuous. We also define $Mod(\Lambda)$ to be the category of (abstract) left $\Lambda$-modules. The corresponding categories of right $\Lambda$-modules can and will be identified with categories of left $\Lambda^{op}$-modules.

It is well known that $PMod(\Lambda)$, $DMod(\Lambda)$ and $Mod(\Lambda)$ are abelian categories, that $PMod(\Lambda)$ and $Mod(\Lambda)$ have enough projectives, and that $DMod(\Lambda)$ and $Mod(\Lambda)$ have enough injectives (see \cite[Proposition 5.4.2, Proposition 5.4.4]{R-Z}). So we can apply the results of \cite{Myself} to additive functors, and their derived functors, on these categories. 

We write $\Hom_\Lambda(A,B)$ for $mor(A,B)$, where $A, B \in TMod(\Lambda)$. Note that $\Hom_\Lambda(A,B)$ is naturally an $R$-module. In addition, it will often be given the compact-open topology: we define the sets $$O_{K,U} = \{f \in \Hom_\Lambda(A,B): f(K) \subseteq U\}$$ to be open, whenever $K \subseteq A$ is compact and $U \subseteq B$ is open. Then the $O_{K,U}$ form a subbase for the topology. Note that $\Hom_\Lambda(A,B)$ is discrete for $A \in PMod(\Lambda)$, $B \in DMod(\Lambda)$, and profinite for $A \in DMod(\Lambda)$, $B \in PMod(\Lambda)$ -- see \cite[Lemma 5.1.4]{R-Z}.

We say $A$ is \emph{of type $\FP_n$ over $\Lambda$}, $n \leq \infty$, for $A \in PMod(\Lambda)$, if it has a projective resolution which is finitely generated for the first $n$ steps, and write $PMod(\Lambda)_n$ for the full subcategory of $PMod(\Lambda)$ whose objects are of type $\FP_n$. So $A \in PMod(\Lambda)_0$ if and only if $A$ is finitely generated. We let $PMod(\Lambda)_{-1} = PMod(\Lambda)$. If $A \in PMod(\Lambda)_0$, $B \in PMod(\Lambda)$, $\Hom_\Lambda(A,B)$ is profinite by \cite[(3.7.1)]{S-W}.

We can now define the main functors which will be needed in this paper. All are additive.

First, the forgetful functor $$U: PMod(\Lambda) \rightarrow Mod(U(\Lambda))$$ which forgets the topology but retains the algebraic structure; we will also write $U$ for the same forgetful functor $PMod(R) \rightarrow Mod(U(R))$. This $U$ is clearly exact.

Second, the \emph{completed tensor product} $$\hat{\otimes}_\Lambda: PMod(\Lambda^{op}) \times PMod(\Lambda) \rightarrow PMod(R)$$ -- see \cite[Chapters 5.5, 6.1]{R-Z} for definitions and properties of this right-exact functor -- and its derived functors $$\Tor^\Lambda_\ast: PMod(\Lambda^{op}) \times PMod(\Lambda) \rightarrow PMod(R).$$ We will also need the standard tensor product of abstract modules $$\otimes_{U(\Lambda)}: Mod(U(\Lambda^{op})) \times Mod(U(\Lambda)) \rightarrow Mod(U(R)).$$

The definition of completed tensor products says that there is a unique canonical continuous middle linear map $$B \times A \rightarrow B \hat{\otimes}_\Lambda A,$$ and continuous middle linear maps are clearly middle linear in the abstract sense, so by the universal property of abstract tensor products $$U(B \times A) \rightarrow U(B \hat{\otimes}_\Lambda A)$$ factors canonically (and uniquely) as $$U(B) \times U(A) \rightarrow U(B) \otimes_{U(\Lambda)} U(A) \rightarrow U(B \hat{\otimes}_\Lambda A).$$

The map $$U(B) \otimes_{U(\Lambda)} U(A) \rightarrow U(B \hat{\otimes}_\Lambda A)$$ induces a transformation of functors $$U(-) \otimes_{U(\Lambda)} \rightarrow U (- \hat{\otimes}_\Lambda -)$$ which is natural in both variables, by the universal property of $\otimes_{U(\Lambda)}$.

\begin{lem}
\label{otimes}
Suppose $A \in PMod(\Lambda), B \in PMod(\Lambda^{op})$.
\begin{enumerate}[(i)]
\item If $A$ is finitely generated and projective, the canonical map $$U(B) \otimes_{U(\Lambda)} U(A) \rightarrow U(B \hat{\otimes}_\Lambda A)$$ is an isomorphism.
\item If $A$ is finitely generated, $$U(B) \otimes_{U(\Lambda)} U(A) \rightarrow U(B \hat{\otimes}_\Lambda A)$$ is an epimorphism.
\item If $A$ is finitely presented, $$U(B) \otimes_{U(\Lambda)} U(A) \rightarrow U(B \hat{\otimes}_\Lambda A)$$ is an isomorphism.
\end{enumerate}
Similar results hold for $B$.
\end{lem}
\begin{proof}
\begin{enumerate}[(i)]
\item First suppose $A$ is finitely generated and free. Then the result follows from \cite[Proposition 5.5.3 (b),(c)]{R-Z} that $B \hat{\otimes}_\Lambda \Lambda \cong B$ and $B \hat{\otimes}_\Lambda -$ is additive, so that $U(B \hat{\otimes}_\Lambda \Lambda^n)$ and $U(B) \otimes_{U(\Lambda)} U(\Lambda^n)$ are both isomorphic to $U(B^n)$. Since projectives are summands of frees, the result follows for $A$ finitely generated and projective as well.

\item Consider the short exact sequence $0 \rightarrow K \rightarrow F \rightarrow A \rightarrow 0$ with $F$ free and finitely generated. We get a commutative diagram
\[
\xymatrix{\cdots \ar[r] & U(B) \otimes_{U(\Lambda)} U(F) \ar[r] \ar[d]^\cong & U(B) \otimes_{U(\Lambda)} U(A) \ar[r] \ar[d] & 0 \\
\cdots \ar[r] & U(B \hat{\otimes}_\Lambda F) \ar[r] & U(B \hat{\otimes}_\Lambda A) \ar[r] & 0 \rlap{,}}
\]
and the result follows by the Five Lemma.

\item Consider the short exact sequence $0 \rightarrow K \rightarrow F \rightarrow A \rightarrow 0$ with $F$ free and finitely generated, and $K$ finitely generated. We get a commutative diagram
\[
\xymatrix@C-5pt{\cdots \ar[r] & U(B) \otimes_{U(\Lambda)} U(K) \ar[r] \ar@{->>}[d] & U(B) \otimes_{U(\Lambda)} U(F) \ar[r] \ar[d]^\cong & U(B) \otimes_{U(\Lambda)} U(A) \ar[r] \ar[d] & 0 \\
\cdots \ar[r] & U(B \hat{\otimes}_\Lambda K) \ar[r] & U(B \hat{\otimes}_\Lambda F) \ar[r] & U(B \hat{\otimes}_\Lambda A) \ar[r] & 0 \rlap{,}}
\]
and the result follows by the Five Lemma.
\end{enumerate}
\end{proof}

The third and final functor we will need is $$\Hom_\Lambda(-,-): PMod(\Lambda) \times PMod(\Lambda) \rightarrow Mod(U(R)).$$ If $A \in PMod(\Lambda)_0$, we may also think of $\Hom_\Lambda(-,-)$ as a functor $$PMod(\Lambda)_0 \times PMod(\Lambda) \rightarrow PMod(R),$$ using the compact-open topology. In either case, we can take left derived functors of $\Hom_\Lambda(-,B)$. Explicitly, we can define $$\Ext_\Lambda^\ast(-,B): PMod(\Lambda) \rightarrow Mod(U(R)),$$ and when considering the case where the first variable is of type $\FP_\infty$, we can endow the resulting $R$-modules with a profinite topology to give $$\Ext_\Lambda^\ast(-,B): PMod(\Lambda)_\infty \rightarrow PMod(R).$$ By \cite[Proposition 3.4]{Myself}, the $\Ext_\Lambda^n(-,-)$ are bifunctors $$PMod(\Lambda) \times PMod(\Lambda) \rightarrow Mod(U(R)),$$ and we get long exact sequences in each variable. We will also need the standard $\Hom$-functor of abstract modules, which we will write $$\hom_{U(\Lambda)}(-,-): Mod(U(\Lambda)) \times Mod(U(\Lambda)) \rightarrow Mod(U(R)).$$ There is a canonical natural transformation $$\Hom_\Lambda(-,-) \rightarrow \hom_{U(\Lambda)}(U(-),U(-))$$ with each $$\Hom_\Lambda(A,B) \rightarrow \hom_{U(\Lambda)}(U(A),U(B))$$ given by the inclusion of the group of continuous homomorphisms into the group of abstract homomorphisms (after forgetting the topology if we are considering $\Hom_\Lambda(A,B)$ as a topological $R$-module).

\begin{lem}
\label{hom}
Suppose $A \in PMod(\Lambda), B \in PMod(\Lambda)$. If $A$ is finitely generated, the canonical map $$\Hom_\Lambda(A,B) \rightarrow \hom_{U(\Lambda)}(U(A),U(B))$$ is an isomorphism.
\end{lem}
\begin{proof}
See \cite[Lemma 7.2.2]{Wilson}.
\end{proof}

\section{Direct Systems of Profinite Modules}
\label{dspm}

Let $I$ be a \emph{directed poset}: a poset with the property that for any $i_1,i_2 \in I$ there is some $i \in I$ such that $i \geq i_1,i_2$. It is easy to check that if $I$ and $J$ are directed posets, $I \times J$ is again directed, when we define $(i_1,j_1) \leq (i_2,j_2)$ if and only if $i_1 \leq i_2$ and $j_1 \leq j_2$. We can define a category $I'$ whose objects are the elements of $I$ and which has a single morphism $i \rightarrow j$ whenever $i \leq j$. Then a covariant functor from $I'$ to some other category $C$ is exactly the same thing as a \emph{direct system} in $C$ indexed by $I$ -- and similarly contravariant functors correspond to inverse systems. Henceforth, directed posets and direct systems will be identified with the corresponding categories and functors.

In the terminology of \cite{Myself}, given a category $C$ and a small category $I$, we can define the functor category $C^I$ as the category of functors $I \rightarrow C$ and natural transformations between them. We can also define, for a functor $F: C \rightarrow D$, the exponent functor $F^I: C^I \rightarrow D^I$: given $f \in C^I$, set $F^I(f)(i) = F(f(i))$, and similarly for morphisms. Henceforth, given any functor $F$ and small category $I$, $F^I$ will denote this exponent functor. In this paper we will be interested in the functor category $PMod(\Lambda)^I$ when $I$ is a directed poset, and in the exponents of the functors defined in the previous section.

In particular we have the functor $$U^I: PMod(\Lambda)^I \rightarrow Mod(U(\Lambda))^I$$ which forgets the topology on each module in a directed system in $PMod(\Lambda)^I$. Now $U$ is exact, so $U^I$ is also exact, by \cite[Lemma 1.9]{Myself}. Second, we have the \emph{direct limit} functor $\varinjlim$ which sends a directed system of (abstract) $\Lambda$-modules to their colimit in the category of (abstract) $\Lambda$-modules. It is well known that for a directed poset $I$ $\varinjlim$ is an exact additive functor $Mod(U(\Lambda))^I \rightarrow Mod(U(\Lambda))$ and $Mod(U(R))^I \rightarrow Mod(U(R))$. So we can compose these two exact functors; it follows that their composition $$\varinjlim U^I: PMod(\Lambda)^I \rightarrow Mod(U(\Lambda)),$$ which forgets the topology on a direct system of modules and then takes its direct limit, is exact. Thus we can compose this (composite) functor with a homological $\delta$-functor to get another homological $\delta$-functor, because long exact sequences are preserved. 

By \cite[Proposition 3.3]{Myself}, we have a long exact sequence in each variable of the exponent homological $\delta$-functor $$\Tor^{\Lambda, I \times J}_\ast: PMod(\Lambda^{op})^I \times PMod(\Lambda)^J \rightarrow PMod(R)^{I \times J},$$ for any posets $I$ and $J$. By \cite[Proposition 3.5]{Myself}, we have a long exact sequence in each variable of $$\Ext_\Lambda^{\ast, I \times J}: PMod(\Lambda)^I \times PMod(\Lambda)^J \rightarrow Mod(U(R))^{I \times J}.$$ When $J$ consists of a single element we may write $\Tor^{\Lambda, I}_\ast$ and $\Ext_\Lambda^{\ast, I}$; similarly in the other variable.

We can now start proving some results. For our main result of the section, we need this preliminary lemma, whose proof is an easy adaptation of \cite[Lemma 2]{Shannon}.

\begin{lem}
\label{directsystem}
For every profinite module $B \in PMod(\Lambda^{op})$, there is a direct system $\{B^i\}$ of finitely presented modules in $PMod(\Lambda^{op})$ with a collection of continuous compatible maps $B^i \rightarrow B$ such that the induced map $$\varinjlim U(B^i) \rightarrow U(B)$$ is an isomorphism.
\end{lem}
\begin{proof}
Let $F$ be the free profinite right $\Lambda$-module with basis $B$. By the universal property of free modules, the identity map $B \rightarrow B$ extends to a canonical continuous homomorphism of profinite modules $F \rightarrow B$. Consider the set of all pairs $(F_S, V)$ where $S$ is a finite subset of $B$, $F_S$ is the free profinite submodule of $F$ generated by $S$ and $V$ is a finitely generated profinite submodule of $F$ such that the composite $$V \hookrightarrow F_S \rightarrow B$$ is the zero map. Define a partial order on this set by $$(F_S, V) \leq (F_T, W) \Leftrightarrow S \subseteq T \text{ and } V \subseteq W.$$ This is clearly directed, so we get a direct system of finitely presented profinite modules $F_S/V$ with the canonical continuous module homomorphisms between them, and canonical compatible continuous module homomorphisms $F_S/V \rightarrow B$. Forgetting the topology by applying $U$, we get a direct system of abstract modules with a compatible collection of module homomorphisms $$f_{S,V}: U(F_S/V) \rightarrow U(B),$$ and hence a module homomorphism $$f: \varinjlim U(F_S/V) \rightarrow U(B).$$ We claim $f$ is an isomorphism. Given $b \in B$, $b$ is in the image of $$f_{\{b\},0}: U(F_{\{b\}}) \rightarrow U(B),$$ and hence it is in the image of $f$. So $f$ is surjective. Given $x$ in the kernel of $f$, take a representative $x'$ of $x$ in one of the $U(F_S/V)$, so $f_{S,V}(x')=0$, and a representative $x''$ of $x'$ in $U(F_S)$. Now suppose $V$ is generated by $x_1, \ldots, x_n$. Let $V'$ be the profinite submodule of $F_S$ generated by $x_1, \ldots, x_n, x''$, so that $V'$ is finitely generated. Note that the composite $$V' \hookrightarrow F_S \rightarrow B$$ is the zero map, and that $(F_S, V) \leq (F_S, V')$. Finally, note that the image of $x'$ in $F_S/V'$ is $0$, and hence the image of $x'$ in $\varinjlim U(F_S/V)$ is $0$, so $x=0$. So $f$ is injective.
\end{proof}

Note that this result is weaker than saying that $B$ can be written as a direct limit of finitely presented profinite modules; indeed, taking the direct limit in $TMod(\Lambda^{op})$ of the system of profinite modules described above will not in general give a profinite module. This will be a recurring theme throughout the paper: that we are required to consider certain direct systems of profinite modules whose direct limits as topological modules have underlying abstract modules that are isomorphic, but may not have the same topology. It is in this way that the following theorem is the profinite analogue of \cite[Theorem 1.1.3]{Bieri}.

\begin{thm}
\label{b-ehom}
Suppose $A \in PMod(\Lambda)$. The following are equivalent:
\begin{enumerate}[(i)]
\item $A \in PMod(\Lambda)_n$.
\item If $I$ is a directed poset, and $B,C \in PMod(\Lambda^{op})^I$, with a morphism $f: B \rightarrow C$ such that $$\varinjlim U^I (f): \varinjlim U^I (B) \rightarrow \varinjlim U^I (C)$$ is an isomorphism, then the induced maps
\begin{align*}
\varinjlim U^I \Tor^{\Lambda, I}_m(f): &\varinjlim U^{I} \Tor^{\Lambda, I}_m (B,A) \\
\rightarrow &\varinjlim U^I \Tor^{\Lambda, I}_m (C,A)
\end{align*}
are isomorphisms for $m<n$ and an epimorphism for $m=n$.
\item For all products $\prod \Lambda$ of copies of $\Lambda$, (ii) holds when $C$ has as each of its components $\prod \Lambda$, with identity maps between them, for some $B$ with each component finitely presented.
\item If $I$ is a directed poset, and $B,C \in PMod(\Lambda)^I$, with a morphism $f: B \rightarrow C$ such that $$\varinjlim U^I (f): \varinjlim U^I (B) \rightarrow \varinjlim U^I (C)$$ is an isomorphism, then the induced maps
\begin{align*}
\varinjlim U^I \Ext_\Lambda^{m, I}(f): &\varinjlim U^{I} \Ext_\Lambda^{m, I} (A,B) \\
\rightarrow &\varinjlim U^{I} \Ext_\Lambda^{m, I} (A,C)
\end{align*}
are isomorphisms for $m<n$ and a monomorphism for $m=n$.
\item (iv) holds when $C$ has $0$ as each of its components, for some $B$ with each component finitely presented.
\end{enumerate}
\end{thm}
\begin{proof}
(i) $\Rightarrow$ (ii): Take a projective resolution $P_\ast$ of $A$ with $P_0, \ldots P_n$ finitely generated. Then for each $i \in I$ we get a diagram

\scalebox{0.83}{\begin{xy}
\xymatrix@C=0pt{\cdots \ar[rr]^(0.3){d_1^i} & & U(B^i) \otimes_{U(\Lambda)} U(P_1) \ar[rr]^{d_0^i} \ar[dd]^(0.3){\alpha_1^i} \ar[dr]^{\gamma_1^i} & & U(B^i) \otimes_{U(\Lambda)} U(P_0) \ar[rr] \ar[dd]^(0.3){\alpha_0^i} \ar[dr]^{\gamma_0^i} & & 0 & \\
& \cdots \ar[rr]^(0.3){e_1^i} & & U(C^i) \otimes_{U(\Lambda)} U(P_1) \ar[rr]^(0.4){e_0^i} \ar[dd]_(0.3){\beta_1^i} & & U(C^i) \otimes_{U(\Lambda)} U(P_0) \ar[rr] \ar[dd]_(0.3){\beta_0^i} & & 0 \\
\cdots \ar[rr]^(0.4){d_1^{\prime i}} & & U(B^i \hat{\otimes}_\Lambda P_1) \ar[rr]^(0.6){d_0^{\prime i}} \ar[dr]^{\delta_1^i} & & U(B^i \hat{\otimes}_\Lambda P_0) \ar[rr] \ar[dr]^{\delta_0^i} & & 0 & \\
& \cdots \ar[rr]^{e_1^{\prime i}} & & U(C^i \hat{\otimes}_\Lambda P_1) \ar[rr]^{e_0^{\prime i}} & & U(C^i \hat{\otimes}_\Lambda P_0) \ar[rr] & & 0}
\end{xy}}

\noindent where all the squares commute. By Lemma \ref{otimes}, $\alpha_0^i, \ldots, \alpha_n^i, \beta_0^i, \ldots, \beta_n^i$ are isomorphisms. Now apply $\varinjlim$. Since $\otimes_{U(\Lambda)}$ commutes with direct limits, we have a commutative diagram

\scalebox{0.68}{\begin{xy}
\xymatrix@C=0pt{\cdots \ar[rr]^(0.25){\varinjlim d_1^i} & & (\varinjlim U(B^i)) \otimes_{U(\Lambda)} U(P_1) \ar[rr]^{\varinjlim d_0^i} \ar[dd]^(0.3){\varinjlim \alpha_1^i} \ar[dr]^{\varinjlim \gamma_1^i} & & (\varinjlim U(B^i)) \otimes_{U(\Lambda)} U(P_0) \ar[rr] \ar[dd]^(0.3){\varinjlim \alpha_0^i} \ar[dr]^{\varinjlim \gamma_0^i} & & 0 & \\
& \cdots \ar[rr]^(0.2){\varinjlim e_1^i} & & (\varinjlim U(C^i)) \otimes_{U(\Lambda)} U(P_1) \ar[rr]^(0.4){\varinjlim e_0^i} \ar[dd]_(0.3){\varinjlim \beta_1^i} & & (\varinjlim U(C^i)) \otimes_{U(\Lambda)} U(P_0) \ar[rr] \ar[dd]_(0.3){\varinjlim \beta_0^i} & & 0 \\
\cdots \ar[rr]^(0.4){\varinjlim d_1^{\prime i}} & & \varinjlim U(B^i \hat{\otimes}_\Lambda P_1) \ar[rr]^(0.6){\varinjlim d_0^{\prime i}} \ar[dr]^{\varinjlim \delta_1^i} & & \varinjlim U(B^i \hat{\otimes}_\Lambda P_0) \ar[rr] \ar[dr]^{\varinjlim \delta_0^i} & & 0 & \\
& \cdots \ar[rr]^{\varinjlim e_1^{\prime i}} & & \varinjlim U(C^i \hat{\otimes}_\Lambda P_1) \ar[rr]^{\varinjlim e_0^{\prime i}} & & \varinjlim U(C^i \hat{\otimes}_\Lambda P_0) \ar[rr] & & 0 \rlap{.}}
\end{xy}}

\noindent Then as before we have that $\varinjlim \alpha_0^i, \ldots, \varinjlim \alpha_n^i, \varinjlim \beta_0^i, \ldots, \varinjlim \beta_n^i$ are isomorphisms. By hypothesis $\varinjlim U(B^i) = \varinjlim U(C^i)$, so that $\varinjlim \gamma_0^i, \ldots, \varinjlim \gamma_n^i$ are isomorphisms. Hence $\varinjlim \delta_0^i, \ldots, \varinjlim \delta_n^i$ are, and the result follows after taking direct limits over $J$, and then taking homology.

(ii) $\Rightarrow$ (iii) trivial.

(iii) $\Rightarrow$ (i): Induction on $n$. First suppose $n=0$: we want to show $A \in PMod(\Lambda)_0$. Consider the case where each $C^i$ is a direct product of copies of $\Lambda$ indexed by $X$, $\prod_X \Lambda$, for some set $X$ such that there is an injection $\iota: A \rightarrow X$. Note that we could just use the set $A$ itself here, but in Lemma \ref{b-ehom+} below we will make use of the fact that we only need (iii) to hold for some $X$ with an injection $\iota: A \rightarrow X$ to deduce (i), rather than all $X$, as claimed in the statement of the theorem. Now $B \in PMod(\Lambda^{op})^I_1$, so by (i) $\Rightarrow$ (ii) of Lemma \ref{otimes},
\begin{align*}
\varinjlim U^I (B \hat{\otimes}_\Lambda^I A) &= \varinjlim (U^I(B) \otimes_{U(\Lambda)}^I U(A)) \\
&= \varinjlim (U^I(B)) \otimes_{U(\Lambda)} U(A) \\
&= U(\prod_X \Lambda) \otimes_{U(\Lambda)} U(A),
\end{align*}
where $\hat{\otimes}_\Lambda^I$ and $\otimes_{U(\Lambda)}^I$ are the exponent functors of $\hat{\otimes}_\Lambda$ and $\otimes_{U(\Lambda)}$, respectively. By hypothesis,
\begin{align*}
\varinjlim U^I (f\hat{\otimes}_\Lambda^I-): &\varinjlim U^I (B\hat{\otimes}_\Lambda^I A) = U(\prod_X \Lambda) \otimes_{U(\Lambda)} U(A) \\
\rightarrow &\varinjlim U^I (C\hat{\otimes}_\Lambda^I A) = U(\prod_X \Lambda \hat{\otimes}_\Lambda A) = U(\prod_X A)
\end{align*}
is an epimorphism, so there is a $$c \in U(\prod_X \Lambda) \otimes_{U(\Lambda)} U(A)$$ such that $$\varinjlim U^I (f\hat{\otimes}_\Lambda^I-)(c)$$ is the `diagonal' element of $U(\prod_X A)$ whose $\iota(a)$th component is $a$, for each $a \in A$. Now $c$ has the form $$\sum_{k=1}^m (\prod_{x \in X} \lambda_k^x) \otimes a_k$$ for some $\lambda_k^x \in \Lambda$ and $a_k \in A$, so $$\varinjlim U^I (f\hat{\otimes}_\Lambda^I-)(c)$$ has $\iota(a)$th component $$\sum_{k=1}^m \lambda_k^{\iota(a)} a_k = a.$$ So $a_1, \ldots, a_m$ generate $A$.

For $n>0$, suppose (iii) $\Rightarrow$ (i) holds for $n-1$. We get $A$ finitely generated as before, and an exact sequence
\begin{equation*}
0 \rightarrow K \rightarrow F \rightarrow A \rightarrow 0,
\end{equation*}
with $F$ free and finitely generated. Then, using our long exact sequence in the second variable, we get the diagram
\[
\xymatrix@C-7pt{\cdots \ar[r] & \varinjlim U^I \Tor^{\Lambda, I}_n (B,F) \ar[r] \ar[d]^{\cong} & \varinjlim U^I \Tor^{\Lambda, I}_n (B,A) \ar[r] \ar@{->>}[d] & \varinjlim U^I \Tor^{\Lambda, I}_{n-1} (B,K) \ar[d] \\
\cdots \ar[r] & U \Tor^{\Lambda}_n (\prod \Lambda,F) \ar[r] & U \Tor^{\Lambda}_n (\prod \Lambda,A) \ar[r] & U \Tor^{\Lambda}_{n-1} (\prod \Lambda,K) \\
\ar[r] & \varinjlim U^I \Tor^{\Lambda, I}_{n-1} (B,F) \ar[r] \ar[d]^{\cong} & \varinjlim U^I \Tor^{\Lambda, I}_{n-1} (B,A) \ar[r] \ar[d]^{\cong} & \cdots \\
\ar[r] & U \Tor^{\Lambda}_{n-1} (\prod \Lambda,F \ar[r]) & U \Tor^{\Lambda}_{n-1} (\prod \Lambda,A) \ar[r] & \cdots}
\]
whose squares commute; it follows by the five lemma that the map $$\varinjlim U^I \Tor^{\Lambda, I}_m (B,K) \rightarrow U \Tor^{\Lambda}_m (\prod \Lambda,K)$$ is an isomorphism for $m<n-1$, and an epimorphism for $m=n-1$, for all direct products of copies of $\Lambda$. So by hypothesis $K$ is of type $\FP_{n-1}$, so $A$ is of type $\FP_n$.

(i) $\Rightarrow$ (iv): Take a projective resolution $P_\ast$ of each $A$ with $P_0, \ldots P_n$ finitely generated. Then for each $i \in I$ we get a diagram

\scalebox{0.7}{\begin{xy}
\xymatrix@C=0pt{0 \ar[rr]^(0.4){d_1^i} & & \Hom_\Lambda(P_0, B^i) \ar[rr]^{d_0^i} \ar[dd]^(0.3){\alpha_1^i} \ar[dr]^{\gamma_1^i} & & \Hom_\Lambda(P_1, B^i) \ar[rr] \ar[dd]^(0.3){\alpha_0^i} \ar[dr]^{\gamma_0^i} & & \cdots & \\
& 0 \ar[rr]^(0.25){e_1^i} & & \Hom_\Lambda(P_0, C^i) \ar[rr]^(0.4){e_0^i} \ar[dd]_(0.3){\beta_1^i} & & \Hom_\Lambda(P_1, C^i) \ar[rr] \ar[dd]_(0.3){\beta_0^i} & & \cdots \\
0 \ar[rr]^(0.2){d_1^{\prime i}} & & \hom_{U(\Lambda)}(U(P_0), U(B^i)) \ar[rr]^(0.7){d_0^{\prime i}} \ar[dr]^{\delta_1^i} & & \hom_{U(\Lambda)}(U(P_1), U(B^i)) \ar[rr] \ar[dr]^{\delta_0^i} & & \cdots & \\
& 0 \ar[rr]^{e_1^{\prime i}} & & \hom_{U(\Lambda)}(U(P_0), U(C^i)) \ar[rr]^{e_0^{\prime i}} & & \hom_{U(\Lambda)}(U(P_1), U(C^i)) \ar[rr] & & \cdots}
\end{xy}}

\noindent where all the squares commute. By Lemma \ref{hom}, $\alpha_0^i, \ldots, \alpha_n^i, \beta_0^i, \ldots, \beta_n^i$ are isomorphisms. Now apply $\varinjlim$. Since $\hom_{U(\Lambda)}$ commutes with direct limits in the second argument when the first argument is finitely generated and projective (by \cite[Proposition 1.2]{Bieri}), we have a commutative diagram

\scalebox{0.62}{\begin{xy}
\xymatrix@C=0pt{0 \ar[rr]^(0.4){\varinjlim d_1^i} & & \varinjlim \Hom_\Lambda(P_0, B^i) \ar[rr]^{\varinjlim d_0^i} \ar[dd]^(0.3){\varinjlim \alpha_1^i} \ar[dr]^{\varinjlim \gamma_1^i} & & \varinjlim \Hom_\Lambda(P_1, B^i) \ar[rr] \ar[dd]^(0.3){\varinjlim \alpha_0^i} \ar[dr]^{\varinjlim \gamma_0^i} & & \cdots & \\
& 0 \ar[rr]^(0.25){\varinjlim e_1^i} & & \varinjlim \Hom_\Lambda(P_0, C^i) \ar[rr]^(0.4){\varinjlim e_0^i} \ar[dd]_(0.3){\varinjlim \beta_1^i} & & \varinjlim \Hom_\Lambda(P_1, C^i) \ar[rr] \ar[dd]_(0.3){\varinjlim \beta_0^i} & & \cdots \\
0 \ar[rr]^(0.2){\varinjlim d_1^{\prime i}} & & \varinjlim \hom_{U(\Lambda)}(U(P_0), U(B^i)) \ar[rr]^(0.6){\varinjlim d_0^{\prime i}} \ar[dr]^{\varinjlim \delta_1^i} & & \varinjlim \hom_{U(\Lambda)}(U(P_1), U(B^i)) \ar[rr] \ar[dr]^{\varinjlim \delta_0^i} & & \cdots & \\
& 0 \ar[rr]^{\varinjlim e_1^{\prime i}} & & \varinjlim \hom_{U(\Lambda)}(U(P_0), U(C^i)) \ar[rr]^{\varinjlim e_0^{\prime i}} & & \varinjlim \hom_{U(\Lambda)}(U(P_1), U(C^i)) \ar[rr] & & \cdots \rlap{.}}
\end{xy}}

\noindent By hypothesis $\varinjlim U(B^i) = \varinjlim U(C^i)$, and so $\varinjlim \alpha_0^i, \ldots, \varinjlim \alpha_n^i$, $\varinjlim \beta_0^i, \ldots, \varinjlim \beta_n^i$ and $\varinjlim \delta_0^i, \ldots, \varinjlim \delta_n^i$ are all isomorphisms. It follows that $\varinjlim \gamma_0^i, \ldots, \varinjlim \gamma_n^i$ are, and the result follows after taking cohomology.

(iv) $\Rightarrow$ (v) trivial.

(v) $\Rightarrow$ (i): Induction on $n$. First suppose $n=0$: we want to show $A \in PMod(\Lambda)_0$. Consider the case where $B$ is the direct system \{$A/A'$\}, where $A'$ ranges over the finitely generated submodules of $A$, with the natural projection maps between them. We claim that $\varinjlim A/A'= 0$. For this, we need to show that for all $x \in A$, there is some $A'$ such that the image of $x$ under the projection $A \xrightarrow{\pi} A/A'$ is $0$. So take $A'$ to be the submodule of $A$ generated by $x$, and we are done. Hence $$\varinjlim \Ext_\Lambda^0(A,A/A') = \varinjlim \Hom_\Lambda(A, A/A') = 0;$$ in particular, there is some $A'$ for which the projection $$A \xrightarrow{\pi} A/A'$$ is $0$. So $A=A'$ is finitely generated.

For $n>0$, suppose (v) $\Rightarrow$ (i) holds for $n-1$. We get $A$ finitely generated as before, and an exact sequence
\begin{equation*}
0 \rightarrow K \rightarrow F \rightarrow A \rightarrow 0,
\end{equation*}
with $F$ free and finitely generated. Then, using our long exact sequence in the first variable, it follows that $$\varinjlim \Ext_\Lambda^m(K,B^i)=0$$ for $m \leq n-1$, whenever $\varinjlim B^i = 0$. So by hypothesis $K$ is of type $\FP_{n-1}$, so $A$ is of type $\FP_n$.
\end{proof}

In fact the proof shows slightly more. Given $A \in PMod(\Lambda)_{n-1}$, $n \geq 0$, pick an exact sequence
\begin{equation*}
0 \rightarrow M \rightarrow P_{n-1} \rightarrow \cdots \rightarrow P_0 \rightarrow A \rightarrow 0
\end{equation*}
with $P_0, \ldots, P_{n-1}$ finitely generated and projective, and let $X$ be a set such that there is an injection $\iota: M \rightarrow X$.

\begin{lem}
\label{b-ehom+}
Let $I$ be a directed poset, let $C \in PMod(\Lambda^{op})^I$ have $\prod_X \Lambda$ for all its components with identity maps between them, let $B \in PMod(\Lambda^{op})^I_1$ such that $$\varinjlim U^I(B) \rightarrow U(\prod_X \Lambda)$$ is an isomorphism, with $B \rightarrow C$ given by the canonical map on each component. Then $A \in PMod(\Lambda)_n$ if and only if $$\varinjlim U^I \Tor^{\Lambda,I}_{n-1} (B,A) \rightarrow U \Tor^\Lambda_{n-1}(\prod_X \Lambda,A)$$ is an isomorphism and $$\varinjlim U^I \Tor^{\Lambda,I}_n (B,A) \rightarrow U \Tor^\Lambda_n(\prod_X \Lambda,A)$$ is an epimorphism.
\end{lem}

\begin{cor}
\label{fp1}
Suppose $A \in PMod(\Lambda)$. Then $$A \text{ is of type } \FP_1 \Leftrightarrow U(C) \otimes_{U(\Lambda)} U(A) \cong U(C \hat{\otimes}_\Lambda A)$$ for all $C \in PMod(\Lambda^{op})$.
\end{cor}
\begin{proof}
$\Rightarrow$: Lemma \ref{otimes}. $\Leftarrow$: Let $C$ be any product of copies of $\Lambda$, $\prod \Lambda$, which is free, so $$\Tor^{\Lambda}_m (\prod \Lambda,A) = 0$$ for $m \geq 1$. Hence for any direct system $B$ of modules in $PMod(\Lambda^{op})$ and any map $$B \rightarrow C=(\prod \Lambda)_{i \in I}$$ such that $$\varinjlim U^I(B) \rightarrow \varinjlim U^I(C)$$ is an isomorphism, $$\varinjlim U^I \Tor^\Lambda_m(B,A) \rightarrow \varinjlim U^I \Tor^\Lambda_m(C,A)$$ must be an epimorphism. Then our hypothesis gives that $$\varinjlim U^I \Tor^\Lambda_0(B,A) \rightarrow \varinjlim U^I \Tor^\Lambda_0(C,A)$$ is an isomorphism, so $A$ is of type $\FP_1$ by (iii) $\Rightarrow$ (i) of the theorem.
\end{proof}

\begin{rem}
\phantomsection
\label{finpres}
\begin{enumerate}[(a)]
\item Ribes-Zalesskii claim in \cite[Proposition 5.5.3]{R-Z} that $A$ being finitely generated is enough for $$U(B) \otimes_{U(\Lambda)} U(A) \rightarrow U(B \hat{\otimes}_\Lambda A)$$ to be an isomorphism for all $B$. (Their notation is slightly different.) If this were the case, then by Corollary \ref{fp1} every finitely generated $A$ would be of type $\FP_1$, and hence by an inductive argument would be of type $\FP_\infty$ (see Lemma \ref{extend} below). In other words $\Lambda$ would be noetherian, in the sense of \cite{S-W}, for all profinite $\Lambda$. But this isn't true: we will see in Remark \ref{prosoluble}(a) that for a group $G$ in certain classes of profinite groups, including prosoluble groups, if $G$ is infinitely generated then $\hat{\mathbb{Z}}$ is of type $\FP_0$ but not $\FP_1$ considered as a $\hat{\mathbb{Z}} \llbracket G \rrbracket$-module with trivial $G$-action, giving a contradiction.
\item A similar claim to the one in (a) is made by Brumer in \cite[Lemma 2.1(ii)]{Brumer}, where `profinite' is replaced by `pseudocompact'. Since profinite rings and modules are pseudocompact, the argument of (a) shows that Brumer's claim also produces a contradiction.
\end{enumerate}
\end{rem}

\begin{cor}
\label{fpn}
If $1 \leq n < \infty$, the following are equivalent for $A \in PMod(\Lambda)$:
\begin{enumerate}[(i)]
\item $A \in PMod(\Lambda)_n$.
\item If $I$ is a directed poset, $B,C \in PMod(\Lambda^{op})^I$, with a morphism $f: B \rightarrow C$ such that $$\varinjlim U^I (f): \varinjlim U^I (B) \rightarrow \varinjlim U^I (C)$$ is an isomorphism, and each component of $C$ is a product of copies of $\Lambda$ with identity maps between them, then $$\varinjlim U^I (B \hat{\otimes}_\Lambda^I A) \rightarrow U (\prod \Lambda \hat{\otimes}_\Lambda A) = U(\prod A)$$ is an isomorphism and $$\varinjlim U^I \Tor^{\Lambda, I}_m (B,A) = 0$$ for $1 \leq m \leq n-1$.
\item $A \in PMod(\Lambda)_1$ and $$\varinjlim U^I \Tor^{\Lambda, I}_m (B,A) = 0$$ for $1 \leq m \leq n-1$.
\end{enumerate}
\end{cor}
\begin{proof}
Use (i) $\Leftrightarrow$ (iii) from Theorem \ref{b-ehom}. Then (iii) from the theorem $\Leftrightarrow$ (ii) because $$U \Tor^\Lambda_m(\prod \Lambda,A) = 0,$$ for all $m>0$, and (ii) $\Leftrightarrow$ (iii) by Corollary \ref{fp1}.
\end{proof}

As in Lemma \ref{b-ehom+}, suppose we have $A \in PMod(\Lambda)_{n-1}$, $n \geq 0$, pick an exact sequence
\begin{equation*}
0 \rightarrow M \rightarrow P_{n-1} \rightarrow \cdots \rightarrow P_0 \rightarrow A \rightarrow 0
\end{equation*}
with $P_0, \ldots, P_{n-1}$ finitely generated and projective, and let $X$ be a set such that there is an injection $\iota: M \rightarrow X$. Let $I$ be a directed poset, let $C \in PMod(\Lambda^{op})^I$ have $\prod_X \Lambda$ for all its components with identity maps between them, let $B \in PMod(\Lambda^{op})^I_1$ such that $$\varinjlim U^I(B) \rightarrow U(\prod_X \Lambda)$$ is an isomorphism, with $B \rightarrow C$ given by the canonical map on each component.

\begin{cor}
\label{fpn+}
Assume in addition that $n \geq 1$. Then $A \in PMod(\Lambda)_n$ if and only if $$\varinjlim U^I \Tor^{\Lambda,I}_{n-1} (B,A) \rightarrow U \Tor^\Lambda_{n-1}(\prod_X \Lambda,A)$$ is an isomorphism. For $n \geq 2$, $A \in PMod(\Lambda)_n$ if and only if $$\varinjlim U^I \Tor^{\Lambda,I}_{n-1} (B,A) = 0.$$
\end{cor}
\begin{proof}
$U \Tor^\Lambda_n(\prod_X \Lambda,A) = 0$, for all $n>0$.
\end{proof}

Now analogues to other results in \cite[Chapter 1.1]{Bieri} follow directly from this.

\begin{cor}
Suppose $A' \rightarrowtail A \twoheadrightarrow A''$ is an exact sequence in $PMod(\Lambda)$. Then:
\begin{enumerate}[(i)]
\item If $A' \in PMod(\Lambda)_{n-1}$ and $A \in PMod(\Lambda)_n$, then $A'' \in PMod(\Lambda)_n$.
\item If $A \in PMod(\Lambda)_{n-1}$ and $A'' \in PMod(\Lambda)_n$, then $A'$ is of type $\FP_{n-1}$.
\item If $A'$ and $A''$ are $\in PMod(\Lambda)_n$ then so is $A$.
\end{enumerate}
\end{cor}
\begin{proof}
This follows immediately from the long exact sequences in $\Tor^{\Lambda, I}_\ast$.
\end{proof}

\begin{lem}
\label{extend}
Let $A \in PMod(\Lambda)$ be of type $\FP_n$, $n<\infty$, and let
\begin{equation*}
P_{n-1} \rightarrow \cdots \rightarrow P_1 \rightarrow P_0 \rightarrow A \rightarrow 0
\end{equation*}
be a partial projective resolution with $P_0, \ldots, P_{n-1}$ finitely generated. Then the kernel $\ker(P_{n-1} \rightarrow P_{n-2})$ is finitely generated, so one can extend the resolution to
\begin{equation*}
P_n \rightarrow P_{n-1} \rightarrow \cdots \rightarrow P_1 \rightarrow P_0 \rightarrow A \rightarrow 0,
\end{equation*}
with $P_n$ finitely generated as well.
\end{lem}
\begin{proof}
See \cite[Proposition 1.5]{Bieri}.
\end{proof}

\begin{cor}
\label{fpinfty}
Suppose $A \in PMod(\Lambda)$. The following are equivalent:
\begin{enumerate}[(i)]
\item $A \in PMod(\Lambda)_\infty$.
\item If $I$ is a directed poset, $B,C \in PMod(\Lambda^{op})^I$, with a morphism $f: B \rightarrow C$ such that $$\varinjlim U^I (f): \varinjlim U^I (B) \rightarrow \varinjlim U^I (C)$$ is an isomorphism, and each component of $C$ is a product of copies of $\Lambda$ with identity maps between them, then $$\varinjlim U^I (B \hat{\otimes}_\Lambda^I A) \rightarrow U (\prod \Lambda \hat{\otimes}_\Lambda A) = U(\prod A)$$ is an isomorphism and $$\varinjlim U^I \Tor^{\Lambda, I}_m (B,A) = 0$$ for all $m \geq 1$.
\item $A \in PMod(\Lambda)_1$ and $$\varinjlim U^I \Tor^{\Lambda, I}_m (B,A) = 0$$ for all $m \geq 1$.
\item If $I$ is a directed poset, and $B \in PMod(\Lambda)^I$ such that $\varinjlim U^I (B) = 0$, then
\begin{equation*}
\varinjlim U^I \Ext_\Lambda^{m, I} (A,B) = 0
\end{equation*}
for all $m$.
\end{enumerate}
\end{cor}
\begin{proof}
(i) $\Rightarrow$ (ii) $\Rightarrow$ (iii) follows immediately from Corollary \ref{fpn}; for (iii) $\Rightarrow$ (i), Corollary \ref{fpn} shows that $A \in PMod(\Lambda)_n$, for all $n<\infty$, and then Lemma \ref{extend} allows us to construct the required projective resolution of $A$. (i) $\Rightarrow$ (iv) follows from Theorem \ref{b-ehom}, which also shows that (iv) $\Rightarrow$ $A \in PMod(\Lambda)_n$, for all $n<\infty$, and then Lemma \ref{extend} tells us that this implies (i).
\end{proof}

\section{Profinite Group Homology and Cohomology over Direct Systems}
\label{groupdspm}

Let $R$ be a commutative profinite ring and $G$ a profinite group. See \cite[Chapter 5.3]{R-Z} for the definition of the \emph{complete group algebra} $R \llbracket G \rrbracket $. Then for $I$ a small category, $A \in PMod(R \llbracket G \rrbracket ^{op})^I$, $B \in PMod(R \llbracket G \rrbracket)^I$, we define the \emph{homology groups of $G$ over $R$ with coefficients in $A$} by $$H^{R,I}_n(G,A) = \Tor^{R \llbracket G \rrbracket , I}_n(A,R),$$ and the \emph{cohomology groups with coefficients in $B$} by $$H^{n,I}_R(G,B) = \Ext^{n, I}_{R \llbracket G \rrbracket} (R,B),$$ where $R$ is a left $R \llbracket G \rrbracket $-module via the trivial $G$-action.

If $R$ is of type $\FP_n$ as an $R \llbracket G \rrbracket $-module, we say $G$ is \emph{of type $\FP_n$ over $R$}. Note that $R$ is finitely generated as an $R \llbracket G \rrbracket $-module, so all groups are of type $\FP_0$ over all $R$. Note also that since $R \llbracket \{e\} \rrbracket  = R$, $R$ is free as an $R \llbracket \{e\} \rrbracket $-module, so the trivial group is of type $\FP_\infty$.

Now Theorem \ref{b-ehom} and Corollary \ref{fpn} translate to:

\begin{prop}
\label{groupfpn}
Let $I$ be a directed poset. The following are equivalent for $n \geq 1$:
\begin{enumerate}[(i)]
\item $G$ is of type $\FP_n$ over $R$.
\item Whenever we have $B,C \in PMod(R \llbracket G \rrbracket ^{op})^I$, with a morphism $f: B \rightarrow C$ such that $$\varinjlim U^I (f): \varinjlim U^I (B) \rightarrow \varinjlim U^I (C)$$ is an isomorphism, then $$\varinjlim U^I H^{R,I}_m(G,B) \rightarrow \varinjlim U^I H^{R,I}_m(G,C)$$ are isomorphisms for $m<n$ and an epimorphism for $m=n$.
\item $G$ is of type $\FP_1$, and for all products $\prod \Lambda$ of copies of $R \llbracket G \rrbracket $, when $C$ has as each of its components $\prod \Lambda$, with identity maps between them, for some $B$ with each component finitely presented, $$\varinjlim U^I H^{R,I}_m(G,B) = 0$$ for all $1 \leq m \leq n-1$.
\item Whenever we have $B,C \in PMod(R \llbracket G \rrbracket )^I$, with a morphism $f: B \rightarrow C$ such that $$\varinjlim U^I (f): \varinjlim U^I (B) \rightarrow \varinjlim U^I (C)$$ is an isomorphism, then $$\varinjlim U^I H_R^{m, I} (G,B) \rightarrow \varinjlim U^I H_R^{m, I} (G,C)$$ are isomorphisms for $m<n$ and a monomorphism for $m=n$.
\item When $C$ has $0$ as each of its components, for some $B$ with each component finitely presented, $$\varinjlim U^I H_R^{m, I} (G,B) = 0$$ for $m \leq n$.
\end{enumerate}
\end{prop}

Similar results hold for $n=\infty$, by Lemma \ref{extend}.

Corollary \ref{fpn+} translates to:

\begin{lem}
\label{groupfpn+}
Suppose $G$ is of type $\FP_{n-1}$, $n \geq 1$, and we have an exact sequence
\begin{equation*}
0 \rightarrow M \rightarrow P_{n-1} \rightarrow \cdots \rightarrow P_0 \rightarrow R \rightarrow 0
\end{equation*}
of profinite left $R \llbracket G \rrbracket $-modules with $P_0, \ldots, P_{n-1}$ finitely generated and projective. Let $I$ be a directed poset, let $C \in PMod(R \llbracket G \rrbracket ^{op})^I$ have $\prod_X R \llbracket G \rrbracket $ for all its components with identity maps between them, for a set $X$ such that there is an injection $\iota: M \rightarrow X$, let $B \in PMod(R \llbracket G \rrbracket ^{op})^I_1$ such that $$\varinjlim U^I(B) \rightarrow U(\prod_X R \llbracket G \rrbracket )$$ is an isomorphism, with $B \rightarrow C$ given by the canonical map on each component. Then $G$ is of type $\FP_n$ if and only if $$\varinjlim U^I H^{R,I}_{n-1} (G,B) \rightarrow U H^R_{n-1}(G,\prod_X R \llbracket G \rrbracket )$$ is an isomorphism.

For $n \geq 2$, $G$ is of type $\FP_n$ if and only if $$\varinjlim U^I H^{R,I}_{n-1} (G,B) = 0.$$
\end{lem}

\begin{lem}
\label{opensbgp}
Suppose $H$ is an open subgroup of $G$. Then $H$ is of type $\FP_n$ over $R$, $n \leq \infty$, if and only if $G$ is. In particular, if $G$ is finite, it is of type $\FP_\infty$ over $R$.
\end{lem}
\begin{proof}
$H$ open $\Rightarrow$ $H$ is of finite index in $G$. It follows from \cite[Proposition 5.7.1]{R-Z} that $R \llbracket G \rrbracket $ is free and finitely generated as an $R \llbracket H \rrbracket $-module, and hence that a finitely generated projective $R \llbracket G \rrbracket $-module is also a finitely generated projective $R \llbracket H \rrbracket $-module (because projective modules are summands of free ones). So an $R \llbracket G \rrbracket $-projective resolution of $R$, finitely generated up to the $n$th step, shows that $H$ is of type $\FP_n$.

For the converse, suppose $H$ is of type $\FP_n$, and suppose we have a finitely generated partial $R \llbracket G \rrbracket $-projective resolution
\begin{equation*}
P_k \rightarrow \cdots \rightarrow P_0 \rightarrow R \rightarrow 0, \tag{$\ast$}
\end{equation*}
for $k<n$. Then since ($\ast$) is also a finitely generated partial $R \llbracket H \rrbracket $-projective resolution, $\ker(P_k \rightarrow P_{k-1})$ is finitely generated as an $R \llbracket H \rrbracket $-module, by Lemma \ref{extend}. So it is finitely generated as an $R \llbracket G \rrbracket $-module too. So we can extend the $R \llbracket G \rrbracket $-projective resolution to
\begin{equation*}
P_{k+1} \rightarrow P_k \rightarrow \cdots \rightarrow P_0 \rightarrow R \rightarrow 0,
\end{equation*}
with $P_{k+1}$ finitely generated. Iterate this argument to get that $G$ is of type $\FP_n$.
\end{proof}

We now observe that if a group $G$ is of type $\FP_n$ over $\hat{\mathbb{Z}}$, it is of type $\FP_n$ over all profinite $R$ (see \cite[Lemma 6.3.5]{R-Z}). Indeed, given a partial projective resolution
\begin{equation*}
P_n \rightarrow P_{n-1} \rightarrow \cdots \rightarrow P_0 \rightarrow \hat{\mathbb{Z}} \rightarrow 0
\end{equation*}
of $\hat{\mathbb{Z}}$ as a $\hat{\mathbb{Z}}\llbracket G \rrbracket$-module with each $P_k$ finitely generated, apply $-\hat{\otimes}_{\hat{\mathbb{Z}}}R$: this is exact because the resolution is $\hat{\mathbb{Z}}$-split. Trivially $\hat{\mathbb{Z}} \hat{\otimes}_{\hat{\mathbb{Z}}}R \cong R$. Now $\hat{\mathbb{Z}}\llbracket G \rrbracket \hat{\otimes}_{\hat{\mathbb{Z}}}R = R\llbracket G \rrbracket$ by considering inverse limits of finite quotients, and it follows by additivity that each $P_k \hat{\otimes}_{\hat{\mathbb{Z}}}R$ is a finitely generated projective $R\llbracket G \rrbracket$-module, as required.

For a profinite group $G$, we write $d(G)$ for the minimal cardinality of a set of generators of $G$. For a profinite $\hat{\mathbb{Z}} \llbracket G \rrbracket$-module $A$, $d_{\hat{\mathbb{Z}} \llbracket G \rrbracket}(A)$ is the minimal cardinality of a set of generators of $A$ as a $\hat{\mathbb{Z}} \llbracket G \rrbracket$-module. Similarly for abstract groups -- except that we count abstract generators instead of topological generators.

We define the \emph{augmentation ideal} $I_{\hat{\mathbb{Z}}} \llbracket G \rrbracket $ to be the kernel of the \emph{evaluation map}
\begin{equation*}
\varepsilon: \hat{\mathbb{Z}} \llbracket G \rrbracket  \rightarrow \hat{\mathbb{Z}}, g \mapsto 1,
\end{equation*}
and $I_{\mathbb{Z}}[G]$ similarly for abstract groups. In the abstract case, $d(G)$ is finite if and only if $d_{\mathbb{Z}[G]}(I_{\mathbb{Z}}[G])$ is, and more generally groups are of type $\FP_1$ over any ring if and only if they are finitely generated, by \cite[Proposition 2.1]{Bieri}. Similarly pro-$p$ groups are of type $\FP_1$ over $\mathbb{Z}_p$ if and only if they are finitely generated, by \cite[Theorem 7.8.1]{R-Z} and \cite[Proposition 4.2.3]{S-W}. The following proposition shows this is no longer the case for profinite groups.

\begin{prop}
\label{fingen}
Let $G$ be a profinite group. Then the following are equivalent.
\begin{enumerate}[(i)]
\item $G$ is finitely generated.
\item There exists some $d$ such that for all open normal subgroups $K$ of $G$, $$d(G/K) \leq d_{\mathbb{Z}[G/K]}(I_{\mathbb{Z}}[G/K]) + d,$$ and $G$ is of type $\FP_1$ over $\hat{\mathbb{Z}}$.
\end{enumerate}
\end{prop}
\begin{proof}
We start by noting:
\begin{enumerate}[(a)]
\item $d(G) = \sup_K d(G/K)$ by \cite[Lemma 2.5.3]{R-Z};
\item $d_{\hat{\mathbb{Z}} \llbracket G \rrbracket }(I_{\hat{\mathbb{Z}}} \llbracket G \rrbracket ) = \sup_K d_{\mathbb{Z}[G/K]}(I_{\mathbb{Z}}[G/K])$ by \cite[Theorem 2.3]{Damian}.
\end{enumerate}
(i) $\Rightarrow$ (ii): For a finitely generated abstract group $G$,
\begin{equation*}
d(G) \geq d_{\mathbb{Z}[G]}(I_{\mathbb{Z}}[G]). \tag{$\ast$}
\end{equation*}
Indeed, if $G$ is generated by $x_1, \ldots, x_k$, then one can check that $I_{\mathbb{Z}}[G]$ is generated as a $\mathbb{Z}[G]$-module by $x_1-1, \ldots, x_k-1$. Write $G$ as the inverse limit of $\{G/K\}$, where $K$ ranges over the open normal subgroups of $G$. Then applying ($\ast$), for each $K$ $$d(G/K) \geq d_{\mathbb{Z}[G/K]}(I_{\mathbb{Z}}[G/K]);$$ hence $$d(G) = \sup_K d(G/K) \geq \sup_K d_{\mathbb{Z}[G/K]}(I_{\mathbb{Z}}[G/K]) = d_{\hat{\mathbb{Z}} \llbracket G \rrbracket }(I_{\hat{\mathbb{Z}}} \llbracket G \rrbracket ),$$ and hence $G$ is of type $\FP_1$ over $\hat{\mathbb{Z}}$. Now set $d=d(G)$: for each $K$, $$d(G/K) \leq d \leq d_{\mathbb{Z}[G/K]}(I_{\mathbb{Z}}[G/K]) + d.$$
(ii) $\Rightarrow$ (i): First note that by Lemma \ref{extend}, since $G$ is of type $\FP_1$, $d_{\hat{\mathbb{Z}} \llbracket G \rrbracket }(I_{\hat{\mathbb{Z}}} \llbracket G \rrbracket )$ is finite. By (a) and (b), $$d(G) \leq d_{\hat{\mathbb{Z}} \llbracket G \rrbracket }(I_{\hat{\mathbb{Z}}} \llbracket G \rrbracket ) + d,$$ and the result follows.
\end{proof}

\begin{rem}
\phantomsection
\label{prosoluble}
\begin{enumerate}[(a)]
\item When, for example, $G$ is prosoluble or 2-generated, it is known that the condition $$d(G/K) \leq d_{\mathbb{Z}[G/K]}(I_{\mathbb{Z}}[G/K]) + d$$ for all open normal $K$ holds with $d=0$ -- see \cite[Proposition 6.2, Theorem 6.9]{Gruenberg}. Since pro-$p$ groups are pronilpotent, this holds for all pro-$p$ groups. By the Feit-Thompson theorem, it holds for all profinite groups of order coprime to $2$.
\item There are profinite groups $G$ for which the difference between $d(G/K)$ and $d_{\mathbb{Z}[G/K]}(I_{\mathbb{Z}}[G/K])$ is unbounded as $K$ varies. The existence of a group of type $\FP_1$ over $\hat{\mathbb{Z}}$ that is not finitely generated is shown in \cite[Example 2.6]{Damian}.
\item Let $\pi$ be a set of primes. In fact the proof of \cite[Theorem 2.3]{Damian} that $$d_{\hat{\mathbb{Z}} \llbracket G \rrbracket }(I_{\hat{\mathbb{Z}}} \llbracket G \rrbracket ) = \sup_K d_{\mathbb{Z}[G/K]}(I_{\mathbb{Z}}[G/K]),$$ and hence the proof of Proposition \ref{fingen}, go through unchanged if $G$ is a pro-$\pi$ group and we replace $\hat{\mathbb{Z}}$ with $\mathbb{Z}_{\hat{\pi}}$, or more particularly if $G$ is pro-$p$ and we use $\mathbb{Z}_p$. Thus, applying (a), we recover in a new way the fact that pro-$p$ groups are finitely generated if and only if they are of type $\FP_1$ over $\mathbb{Z}_p$.
\end{enumerate}
\end{rem}

\begin{cor}
Suppose $G$ is prosoluble or 2-generated profinite. Then $G$ is of type $\FP_\infty$ over $\hat{\mathbb{Z}}$ if and only if it is finitely generated and whenever $B,C \in PMod(\hat{\mathbb{Z}} \llbracket G \rrbracket ^{op})^I$, with a morphism $f: B \rightarrow C$ such that $$\varinjlim U^I (f): \varinjlim U^I (B) \rightarrow \varinjlim U^I (C)$$ is an isomorphism, and each component of $C$ is a product of copies of $\hat{\mathbb{Z}} \llbracket G \rrbracket $ with identity maps between them, $$\varinjlim U^I H^{R, I}_n(G, B) = 0$$ for all $n \geq 1$.
\end{cor}
\begin{proof}
Proposition \ref{groupfpn} and Proposition \ref{fingen}.
\end{proof}

We have, for $H^{R,I}_\ast$, a Lyndon-Hochschild-Serre spectral sequence for profinite groups.

\begin{thm}
\label{lhs}
Let $G$ be a profinite group, $K$ a closed normal subgroup and suppose $B \in PMod(R \llbracket G \rrbracket ^{op})^I$. Then there exists a spectral sequence $(E^t_{r,s})$ with the property that $$E^2_{r,s} \cong H^{R,I}_r(G/K,H^{R,I}_s(K,B))$$ and $$E^2_{r,s} \Rightarrow H^{R,I}_{r+s}(G,B).$$
\end{thm}
\begin{proof}
\cite[Theorem 7.2.4]{R-Z} and \cite[Corollary 2.5]{Myself}.
\end{proof}

\begin{thm}
\label{extensions}
Let $G$ be a profinite group and $K$ a closed normal subgroup. Suppose $K$ is of type $\FP_m$ over $R$, $m \leq \infty$. Suppose $n \leq \infty$, and let $s=\min\{m,n\}$.
\begin{enumerate}[(i)]
\item If $G$ is of type $\FP_n$ over $R$ then $G/K$ is of type $\FP_s$ over $R$.
\item If $G/K$ is of type $\FP_n$ over $R$ then $G$ is of type $\FP_s$ over $R$.
\end{enumerate}
\end{thm}
\begin{proof}
For simplicity we prove the case $m=\infty$. The proof for $m$ finite is similar.

Since $K$ is of type $\FP_\infty$, by Proposition \ref{groupfpn} we have that, whenever $B,C \in PMod(R \llbracket G \rrbracket ^{op})^I$, with a morphism $f: B \rightarrow C$ such that $$\varinjlim U^I (f): \varinjlim U^I (B) \rightarrow \varinjlim U^I (C)$$ is an isomorphism,
\begin{equation*}
\varinjlim U^I H^{R,I}_n(K, B) \rightarrow \varinjlim U^I H^{R,I}_n(K, C)
\end{equation*}
is an isomorphism for all $n$; hence, when the components of $C$ are products of copies of $R \llbracket G \rrbracket $ with identity maps between them,
\begin{equation*}
\varinjlim U^I H^{R,I}_n(K, B) \rightarrow \varinjlim U^I H^{R,I}_n(K,\prod R \llbracket G \rrbracket ) = U(\prod H^R_n(K, R \llbracket G \rrbracket ))
\end{equation*}
is an isomorphism for all $n$; $R \llbracket G \rrbracket $ is a free $R \llbracket K \rrbracket $-module by \cite[Corollary 5.7.2]{R-Z}, so this is $0$ for $n \geq 1$, and for $n=0$ it gives
\begin{equation*}
\varinjlim U^I(B \hat{\otimes}^I_{R \llbracket K \rrbracket } R) = U(\prod R \llbracket G \rrbracket  \hat{\otimes}_{R \llbracket K \rrbracket } R) = U(\prod R \llbracket G/K \rrbracket )
\end{equation*}
by \cite[Proposition 5.8.1]{R-Z}. So the spectral sequence from Theorem \ref{lhs} collapses to give an isomorphism
\begin{equation*}
H^{R,I}_r(G/K,B \hat{\otimes}^I_{R \llbracket K \rrbracket } R) \cong H^{R,I}_r(G,B).
\tag{$\ast$}
\end{equation*}

By Lemma \ref{extend}, it is enough to prove the theorem for $n < \infty$. We use induction on $n$. Note that $G$ and $G/K$ are always both of type $\FP_0$, so we may assume the theorem holds for $n-1$. Suppose $G$ and $G/K$ are of type $\FP_{n-1}$, and that we have exact sequences
\begin{equation*}
0 \rightarrow M \rightarrow P_{n-1} \rightarrow \cdots \rightarrow P_0 \rightarrow R \rightarrow 0
\end{equation*}
of profinite left $R \llbracket G \rrbracket $-modules with $P_0, \ldots, P_{n-1}$ finitely generated and projective, and
\begin{equation*}
0 \rightarrow M' \rightarrow P'_{n-1} \rightarrow \cdots \rightarrow P'_0 \rightarrow R \rightarrow 0
\end{equation*}
of profinite left $R \llbracket G/K \rrbracket $-modules with $P'_0, \ldots, P'_{n-1}$ finitely generated and projective. Choose a set $X$ such that there are injections $\iota: M \rightarrow X$ and $\iota': M' \rightarrow X$. Let $I$ be a directed poset, let $C \in PMod(R \llbracket G \rrbracket ^{op})^I$ have $\prod_X R \llbracket G \rrbracket $ for all its components with identity maps between them, let $B \in PMod(R \llbracket G \rrbracket ^{op})^I_1$ such that $$\varinjlim U^I(B) \rightarrow U(\prod_X R \llbracket G \rrbracket )$$ is an isomorphism, with $B \rightarrow C$ given by the canonical map on each component. Finally, note that $$B \hat{\otimes}^I_{R \llbracket K \rrbracket } R \in PMod(R \llbracket G/K \rrbracket ^{op})^I_1:$$ for each $B^i$, there is an exact sequence $$F_1 \rightarrow F_0 \rightarrow B^i \rightarrow 0$$ with $F_0$ and $F_1$ free and finitely generated $R \llbracket G \rrbracket $-modules, so by the right exactness of $- \hat{\otimes}^I_{R \llbracket K \rrbracket } R$ there is an exact sequence $$F_1 \hat{\otimes}^I_{R \llbracket K \rrbracket } R \rightarrow F_0 \hat{\otimes}^I_{R \llbracket K \rrbracket } R \rightarrow B^i \hat{\otimes}^I_{R \llbracket K \rrbracket } R \rightarrow 0$$ with $F_0 \hat{\otimes}^I_{R \llbracket K \rrbracket } R$ and $F_1 \hat{\otimes}^I_{R \llbracket K \rrbracket } R$ free and finitely generated $R \llbracket G/K \rrbracket $-modules by \cite[Proposition 5.8.1]{R-Z}. Therefore $G$ is of type $\FP_n$ if and only if $$\varinjlim U^I H^{R,I}_{n-1} (G,B) \rightarrow U H^R_{n-1}(G,\prod_X R \llbracket G \rrbracket )$$ is an isomorphism (by Lemma \ref{groupfpn+}) if and only if $$\varinjlim U^I H^{R,I}_{n-1} (G/K,B \hat{\otimes}^I_{R \llbracket K \rrbracket } R) \rightarrow U H^R_{n-1}(G/K,\prod_X R \llbracket G/K \rrbracket )$$ is an isomorphism (by ($\ast$)) if and only if $G/K$ is of type $\FP_n$ (by Lemma \ref{groupfpn+}).
\end{proof}

Let $\mathcal{C}$ be a non-empty \emph{class} of finite groups, i.e. a collection of groups that is closed under isomorphism. Then we can define pro-$\mathcal{C}$ algebraic structures as profinite ones all of whose finite quotients are in $\mathcal{C}$ -- see \cite{R-Z} for details. Suppose $R$ is a pro-$\mathcal{C}$ ring. Then being of type $\FP_n$ over $R$ as a pro-$\mathcal{C}$ group is exactly the same as being of type $\FP_n$ over $R$ as a profinite group, so working in the pro-$\mathcal{C}$ universe instead of the profinite one gives nothing new. On the other hand, amalgamated free pro-$\mathcal{C}$ products of pro-$\mathcal{C}$ groups are not the same as amalgamated free profinite products of pro-$\mathcal{C}$ groups, and pro-$\mathcal{C}$ $\HNN$-extensions of pro-$\mathcal{C}$ groups are not the same as profinite $\HNN$-extensions of pro-$\mathcal{C}$ groups -- essentially because, in the pro-$\mathcal{C}$ case, we take a pro-$\mathcal{C}$ completion of the abstract amalgamated free product or abstract $\HNN$-extension, rather than taking a profinite completion of them. Thus, by working over a class $\mathcal{C}$, we can achieve more general results.

In the abstract case, Bieri uses his analogous results to give conditions on the $\FP$-type of amalgamated free products and $\HNN$-extensions of groups using the Mayer-Vietoris sequence on their homology. His approach does not entirely translate to the pro-$\mathcal{C}$ setting, but we obtain some partial results.

See \cite[9.2]{R-Z} for the definition of amalgamated free products in the pro-$\mathcal{C}$ case, and \cite[9.4]{R-Z} for $\HNN$-extensions. We say that an amalgamated free pro-$\mathcal{C}$ product $G=G_1 \amalg_H G_2$ is \emph{proper} if the canonical homomorphisms $G_1 \rightarrow G$ and $G_2 \rightarrow G$ are monomorphisms. Similarly, we say that a pro-$\mathcal{C}$ $\HNN$-extension $G=\HNN(H,A,f)$ is \emph{proper} if the canonical homomorphism $H \rightarrow G$ is a monomorphism.

Suppose, for the rest of the section, that $\mathcal{C}$ is closed under taking subgroups, quotients and extensions. For example, $\mathcal{C}$ could be all finite groups, or all finite $p$-groups -- or, for example, all finite soluble $\pi$-groups, where $\pi$ is a set of primes. Suppose $R$ is a pro-$\mathcal{C}$ ring.

\begin{prop}
\label{amalg}
Let $G=G_1 \amalg_H G_2$ be a proper amalgamated free pro-$\mathcal{C}$ product of pro-$\mathcal{C}$ groups. Suppose $B$ is a profinite right $R \llbracket G \rrbracket $-module. Then there is a long exact sequence of profinite $R$-modules
\begin{align*}
\cdots &\rightarrow H^R_{n+1}(G,B) \rightarrow H^R_n(H,B) \rightarrow H^R_n(G_1,B) \oplus H^R_n(G_2,B) \\
&\rightarrow H^R_n(G,B) \rightarrow \cdots \rightarrow H^R_0(G,B) \rightarrow 0,
\end{align*}
which is natural in $B$.
\end{prop}
\begin{proof}
See \cite[Proposition 9.2.13]{R-Z} for the long exact sequence. Naturality follows by examining the maps involved.
\end{proof}

\begin{prop}
\label{amalg'}
Let $G=G_1 \amalg_H G_2$ be a proper amalgamated free pro-$\mathcal{C}$ product of pro-$\mathcal{C}$ groups. Suppose $B \in PMod(R \llbracket G \rrbracket ^{op})^I$. Then there is a long exact sequence in $PMod(R)^I$
\begin{align*}
\cdots &\rightarrow H^{R,I}_{n+1}(G,B) \rightarrow H^{R,I}_n(H,B) \rightarrow H^{R,I}_n(G_1,B) \oplus H^{R,I}_n(G_2,B) \\
&\rightarrow H^{R,I}_n(G,B) \rightarrow \cdots \rightarrow H^{R,I}_0(G,B) \rightarrow 0,
\end{align*}
which is natural in $B$.
\end{prop}
\begin{proof}
This follows immediately from the naturality of the long exact sequence in Proposition \ref{amalg}.
\end{proof}

We can now give a result analogous to the first part of \cite[Proposition 2.13 (1)]{Bieri}.

\begin{prop}
Let $G=G_1 \amalg_H G_2$ be a proper amalgamated free pro-$\mathcal{C}$ product of pro-$\mathcal{C}$ groups. If $G_1$ and $G_2$ are of type $\FP_n$ over $R$ and $H$ is of type $\FP_{n-1}$ over $R$ then $G$ is of type $\FP_n$ over $R$.
\end{prop}
\begin{proof}
Take $C$ as in Proposition \ref{groupfpn} to have as each component a product of copies of $R \llbracket G \rrbracket $, with identity maps between the components. Apply Proposition \ref{groupfpn} to the long exact sequence in Proposition \ref{amalg'}. Then the Five Lemma gives the result, by Proposition \ref{groupfpn}.
\end{proof}

See \cite[Chapter 3.3]{R-Z} for the construction and properties of free pro-$\mathcal{C}$ groups.

\begin{cor}
\label{fgfree}
Finitely generated free pro-$\mathcal{C}$ groups are of type $\FP_\infty$ over all pro-$\mathcal{C}$ rings $R$.
\end{cor}
\begin{proof}
Unamalgamated free pro-$\mathcal{C}$ products are always proper by \cite[Corollary 9.1.4]{R-Z}.
\end{proof}

For proper profinite $\HNN$-extensions of profinite groups, we also have a Mayer-Vietoris sequence which is natural in the second variable -- see \cite[Proposition 9.4.2]{R-Z}. It immediately follows in the same way as for Proposition \ref{amalg'} that we get a long exact sequence over a functor category.

\begin{prop}
\label{hnn}
Let $G=\HNN(H,A,f)$ be a proper pro-$\mathcal{C}$ $\HNN$-extension of pro-$\mathcal{C}$ groups. Suppose $B \in PMod(R \llbracket G \rrbracket ^{op})^I$. Then there is a long exact sequence in $PMod(R)^I$
\begin{align*}
\cdots &\rightarrow H^{R,I}_{n+1}(G,B) \rightarrow H^{R,I}_n(A,B) \rightarrow H^{R,I}_n(H,B) \\
&\rightarrow H^{R,I}_n(G,B) \rightarrow \cdots \rightarrow H^{R,I}_0(G,B) \rightarrow 0,
\end{align*}
which is natural in $B$.
\end{prop}

From this, we can get in exactly the same way as for free products with amalgamation a result for $\HNN$-extensions, corresponding to the first part of \cite[Proposition 2.13 (2)]{Bieri}.

\begin{prop}
\label{hnn'}
Let $G=\HNN(H,A,f)$ be a proper pro-$\mathcal{C}$ $\HNN$-extension of pro-$\mathcal{C}$ groups. If $H$ is of type $\FP_n$ over $R$ and $A$ is of type $\FP_{n-1}$ over $R$ then $G$ is of type $\FP_n$ over $R$.
\end{prop}

\section{Applications}
\label{appl}

\begin{example}
\label{procyclic}
We show that torsion-free procyclic groups are of type $\FP_\infty$ over $R$. See \cite[Chapter 2.7]{R-Z} for the results on procyclic groups that will be needed in this paper. Any procyclic group $G$ is finitely generated, so of type $\FP_1$ by Proposition \ref{fingen}. If $G$ is torsion-free, its Sylow $p$-subgroups are all either 0 or isomorphic to $\mathbb{Z}_p$, so (assuming $G \neq 1$) it is well-known that $G$ has cohomological dimension 1 -- see \cite[Theorem 7.3.1, Theorem 7.7.4]{R-Z}. Consider the short exact sequence
\begin{equation*}
0 \rightarrow \ker \varepsilon \rightarrow R \llbracket G \rrbracket  \xrightarrow{\varepsilon} R \rightarrow 0,
\end{equation*}
where $\varepsilon$ is the evaluation map defined earlier. The kernel $\ker \varepsilon$ is finitely generated by Lemma \ref{extend}. We claim that it is projective -- and hence that our exact sequence is a finitely generated projective resolution of $R$, showing $G$ is of type $\FP_\infty$. To see this, let $A$ be any profinite right $R \llbracket G \rrbracket $-module, and consider the long exact sequence
\begin{align*}
\cdots &\rightarrow \Tor^{R \llbracket G \rrbracket }_1(A, R) \rightarrow \Tor^{R \llbracket G \rrbracket }_0(A, \ker \varepsilon) \\
&\rightarrow \Tor^{R \llbracket G \rrbracket }_0(A, R \llbracket G \rrbracket ) \rightarrow \Tor^{R \llbracket G \rrbracket }_0(A, R) \rightarrow 0.
\end{align*}
Since $R \llbracket G \rrbracket $ is free as an $R \llbracket G \rrbracket $-module, we get $$\Tor^{R \llbracket G \rrbracket }_n(A, R \llbracket G \rrbracket )=0$$ for all $n \geq 1$, and so $$\Tor^{R \llbracket G \rrbracket }_n(A, \ker \varepsilon) \cong \Tor^{R \llbracket G \rrbracket }_{n+1}(A, R)$$ for all $n \geq 1$. Now $G$ has cohomological dimension 1, so $$\Tor^{R \llbracket G \rrbracket }_{n+1}(A, R)=H^R_{n+1}(G,A)=0$$ for $n \geq 1$, so $\ker \varepsilon$ is projective.
\end{example}

We can now use this example to construct some more groups of type $\FP_\infty$.

It is known in the abstract case that polycyclic groups are of type $\FP_\infty$ over $\mathbb{Z}$ (\cite[Examples 2.6]{Bieri}). In the profinite case, it has not been known whether poly-procyclic groups are of type $\FP_\infty$ over $\hat{\mathbb{Z}}$. A result was known for pro-$p$ groups: poly-(pro-$p$-cyclic) groups are shown to be of type $\FP_\infty$ over $\mathbb{Z}_p$ in \cite[Corollary 4.2.5]{S-W}. This proof uses that, for a pro-$p$ group $G$, $H_{\mathbb{Z}_p}^n(G, A)$ is finite for all finite $A \in DMod(\mathbb{Z}_p \llbracket G \rrbracket )$ if and only if $G$ is of type $\FP_\infty$ over $\mathbb{Z}_p$. Indeed, one might expect a similar result to be true for profinite $G$ which only have finitely many primes in their order, but an obstruction to using this method for general profinite groups is that there are infinitely many primes, so one cannot build up to these groups from pro-$p$ ones using the spectral sequence finitely many times. Similarly, although we showed directly that torsion-free procyclic groups are of type $\FP_\infty$ over $R$, there are procyclic groups which are not even virtually torsion-free, in contrast to the pro-$p$ case, as for example the group $\prod_{p \text{ prime} }\mathbb{Z}/p\mathbb{Z}$.

We now define a class of profinite groups: the elementary amenable profinite groups. The definition is entirely analogous to the hereditary definition of elementary amenable abstract groups given in \cite{H-L}. Let $\mathscr{X}_0$ be the class containing only the trivial group, and let $\mathscr{X}_1$ be the class of profinite groups which are (finitely generated abelian)-by-finite. Now define $\mathscr{X}_\alpha$ to be the class of groups $G$ which have a normal subgroup $K$ such that $G/K \in \mathscr{X}_1$ and every finitely generated subgroup of $K$ is in $\mathscr{X}_{\alpha-1}$ for $\alpha$ a successor ordinal. Finally, for $\alpha$ a limit define $\mathscr{X}_\alpha = \bigcup_{\beta < \alpha} \mathscr{X}_\beta$. Then $\mathscr{X} = \bigcup_\alpha \mathscr{X}_\alpha$ is the class of elementary amenable profinite groups. For $G \in \mathscr{X}$ we define the \emph{class} of $G$ to be the least $\alpha$ with $G \in \mathscr{X}_\alpha$.

Note that soluble profinite groups are clearly elementary amenable.

\begin{prop}
Suppose $G$ is an elementary amenable profinite group of finite rank. Then $G$ is of type $\FP_\infty$ over any profinite ring $R$.
\end{prop}
\begin{proof}
By \cite[Theorem 2.7.2]{R-Z}, every procyclic group is a quotient of $\hat{\mathbb{Z}}$ by a torsion-free procyclic group; $\hat{\mathbb{Z}}$ and torsion-free procyclic groups are of type $\FP_\infty$ by Example \ref{procyclic}. Therefore procyclic groups are of type $\FP_\infty$ by Theorem \ref{extensions}, and finitely generated abelian groups are a finite direct sum of procyclic groups by \cite[Proposition 8.2.1(iii)]{Wilson}, so we get that finitely generated abelian groups are of type $\FP_\infty$ by applying Theorem \ref{extensions} finitely many times.

Now use induction on the class of $G$. If $G \in \mathscr{X}_1$, take a finite index abelian $H \leq G$. We have shown $H$ is of type $\FP_\infty$, so $G$ is too by Lemma \ref{opensbgp}. The case where $G$ has class $\alpha$ is trivial for $\alpha$ a limit, so suppose $\alpha$ is a successor. Choose some $K \lhd G$ such that every finitely generated subgroup of $K$ is in $\mathscr{X}_{\alpha-1}$ and $G/K$ is in $\mathscr{X}_1$. Since $G$ is of finite rank, $K$ is finitely generated, so it is in $\mathscr{X}_{\alpha-1}$. By the inductive hypothesis we get that $K$ is of type $\FP_\infty$, and $G/K$ is too so $G$ is by Theorem \ref{extensions}.
\end{proof}

We spend the rest of Section \ref{appl} constructing pro-$\mathcal{C}$ groups of type $\FP_n$ but not of type $\FP_{n+1}$ over $\mathbb{Z}_{\hat{\mathcal{C}}}$ for every $n< \infty$, for $C$ closed under subgroups, quotients and extensions, as before. King in \cite[Theorem F]{King} gives pro-$p$ groups of type $\FP_n$ but not of type $\FP_{n+1}$ over $\mathbb{Z}_p$, but as far as we know the case with $\mathbb{Z}_{\hat{\mathcal{C}}}$ has not been done before for any other class $\mathcal{C}$. Our construction is analogous to \cite[Proposition 2.14]{Bieri}.

Given a profinite space $X$, we can define the \emph{free pro-$\mathcal{C}$ group on $X$}, $F_{\mathcal{C}}(X)$, together with a canonical continuous map $\iota: X \rightarrow F_{\mathcal{C}}(X)$, by the following universal property: for any pro-$\mathcal{C}$ group $G$ and continuous $\phi: X \rightarrow G$, there is a unique continuous homomorphism $\bar{\phi}: F_{\mathcal{C}}(X) \rightarrow G$ such that $\phi = \bar{\phi} \iota$. For a class $C$ closed under subgroups, quotients and extensions, $F_{\mathcal{C}}(X)$ exists for all $X$ by \cite[Proposition 3.3.2]{R-Z}.

Fix $n \geq 0$. Let $\langle x_k,y_k \rangle$ be the free pro-$\mathcal{C}$ group on the two generators $x_k, y_k$, for $1 \leq k \leq n$, and write $D_n$ for their direct product ($D_0$ is the empty product, i.e. the trivial group). Let $F_{\mathbb{Z}_{\hat{\mathcal{C}}}}$ be the free pro-$\mathcal{C}$ group on generators $\{a_l : l \in \mathbb{Z}_{\hat{\mathcal{C}}}\}$, given the usual pro-$\mathcal{C}$ topology. We define a continuous left $D_n$-action on $F_{\mathbb{Z}_{\hat{\mathcal{C}}}}$ in the following way. For each $k$, we have a continuous homomorphism $\langle x_k,y_k \rangle \rightarrow \mathbb{Z}_{\hat{\mathcal{C}}}$ defined by $x_k,y_k \mapsto 1$. Now this gives a continuous $D_n \rightarrow \mathbb{Z}_{\hat{\mathcal{C}}}^n$. Composing this with $n$-fold addition $$\mathbb{Z}_{\hat{\mathcal{C}}}^n \rightarrow \mathbb{Z}_{\hat{\mathcal{C}}}, (a_1, \ldots, a_n) \mapsto a_1 + \cdots + a_n$$ gives a continuous homomorphism $$f: D_n \rightarrow \mathbb{Z}_{\hat{\mathcal{C}}}.$$ Now we can define a continuous action of $D_n$ on $\mathbb{Z}_{\hat{\mathcal{C}}}$ by $$D_n \times \mathbb{Z}_{\hat{\mathcal{C}}} \xrightarrow{f \times id} \mathbb{Z}_{\hat{\mathcal{C}}} \times \mathbb{Z}_{\hat{\mathcal{C}}} \xrightarrow{+} \mathbb{Z}_{\hat{\mathcal{C}}}.$$ Finally, by \cite[Exercise 5.6.2(d)]{R-Z}, this action extends uniquely to a continuous action on $F_{\mathbb{Z}_{\hat{\mathcal{C}}}}$.

Now we can form the semi-direct product $A_n = F_{\mathbb{Z}_{\hat{\mathcal{C}}}} \rtimes D_n$, and by \cite[Exercise 5.6.2(b),(c)]{R-Z} it is a pro-$\mathcal{C}$ group. We record here the universal property of semi-direct products of pro-$\mathcal{C}$ groups; it is a direct translation of the universal property of semi-direct products of abstract groups from \cite[III.2.10, Proposition 27 (2)]{Bourbaki}, which we will need later.

\begin{lem}
\label{semi-direct}
Suppose we have pro-$\mathcal{C}$ groups $N$, $H$ and $K$, with continuous homomorphisms $\sigma: H \rightarrow Aut(N)$ (with the compact-open topology), $f: N \rightarrow K$ and $g: H \rightarrow K$ such that, for all $x \in N$ and $y \in H$, $$g(y)f(x)g(y^{-1}) = f(\sigma(y)(x)).$$ Then there is a unique continuous homomorphism $$h: N \rtimes H \rightarrow K$$ such that $$f = h \circ (N \rightarrow N \rtimes H)$$ and $$g = h \circ (H \rightarrow N \rtimes H).$$
\end{lem}
\begin{proof}
By \cite[III.2.10, Proposition 27 (2)]{Bourbaki} we know there is a unique homomorphism $h: N \rtimes H \rightarrow K$ satisfying these conditions, except that we need to check $h$ is continuous. The proof in \cite{Bourbaki} constructs $h$ as the map $(x,y) \mapsto f(x)g(y)$; this is the composite of the continuous maps (of sets) $$N \times H \rightarrow K \times K \rightarrow K,$$ where the second map is just multiplication in $K$.
\end{proof}

We need two more results about the $A_n$ before we can prove the main proposition. Let $\rtimes \langle x_n \rangle$ be the free pro-$\mathcal{C}$ group generated by $x_n$.

\begin{lem}
For each $n>0$, $F_{\mathbb{Z}_{\hat{\mathcal{C}}}} \rtimes \langle x_n \rangle$ is the free pro-$\mathcal{C}$ group on two generators.
\end{lem}
\begin{proof}
We will show that this group satisfies the requisite universal property. We claim that it is generated by $a_0$ and $x_n$. Clearly allowing $x_n$ (and $x_n^{-1}$) to act on $a_0$ gives $a_l$, for each $l \in \mathbb{Z}$. Now $\{a_l : l \in \mathbb{Z}_{\hat{\mathcal{C}}}\}$ (abstractly) generates a dense subgroup $H$ of $F_{\mathbb{Z}_{\hat{\mathcal{C}}}}$; $\{a_l : l \in \mathbb{Z}\}$ is dense in $\{a_l : l \in \mathbb{Z}_{\hat{\mathcal{C}}}\}$, so it (abstractly) generates a dense subgroup $K$ of $H$; by transitivity of denseness, $K$ is dense in $F_{\mathbb{Z}_{\hat{\mathcal{C}}}}$, so $\{a_l : l \in \mathbb{Z}\}$ topologically generates $F_{\mathbb{Z}_{\hat{\mathcal{C}}}}$.

It remains to show that given a pro-$\mathcal{C}$ $K$ and a map $$f: \{a_0,x_n\} \rightarrow K$$ there is a continuous homomorphism $$g: F_{\mathbb{Z}_{\hat{\mathcal{C}}}} \rtimes \langle x_n \rangle \rightarrow K$$ such that $f = g \iota$, where $\iota$ is the inclusion $\{a_0,x_n\} \rightarrow F_{\mathbb{Z}_{\hat{\mathcal{C}}}} \rtimes \langle x_n \rangle$. Observe, as in \cite[p.91]{R-Z}, that by the universal property of inverse limits it suffices to check the existence of $g$ when $K$ is finite.

To construct $g$, we first note that $f|_{x_n}$ extends uniquely to a continuous homomorphism $$g': \langle x_n \rangle \rightarrow K.$$ Now we define a continuous map of sets $$f': \{a_l: l \in \mathbb{Z}_{\hat{\mathcal{C}}}\} \rightarrow K$$ by $$f'(a_l) = g'(l \cdot x_n) f(a_0) g'(l \cdot x_n)^{-1},$$ where we write $l \cdot x_n$ for the image of $l$ under the obvious isomorphism $\mathbb{Z}_{\hat{\mathcal{C}}} \cong \langle x_n \rangle$; $f'$ extends uniquely to a continuous homomorphism $$g'': F_{\mathbb{Z}_{\hat{\mathcal{C}}}} \rightarrow K.$$ Finally, by the universal property of semi-direct products, Lemma \ref{semi-direct}, we will have the existence of a continuous homomorphism $g$ satisfying the required property as long as $$g'(y) g''(x) g'(y)^{-1} = g''(\sigma(y)(x)),$$ for all $x \in F_{\mathbb{Z}_{\hat{\mathcal{C}}}}$ and $y \in \langle x_n \rangle$, where $\sigma$ is the continuous homomorphism $\langle x_n \rangle \rightarrow Aut(F_{\mathbb{Z}_{\hat{\mathcal{C}}}})$. This is clear by construction.
\end{proof}

By Corollary \ref{fgfree}, $F_{\mathbb{Z}_{\hat{\mathcal{C}}}} \rtimes \langle x_n \rangle$ is now of type $\FP_\infty$ over $\mathbb{Z}_{\hat{\mathcal{C}}}$; hence, by Theorem \ref{extensions}, $$A_{n-1} \rtimes \langle x_n \rangle = (F_{\mathbb{Z}_{\hat{\mathcal{C}}}} \rtimes \langle x_n \rangle) \rtimes D_{n-1}$$ is too.

The next lemma is entirely analogous to \cite[Proposition 2.15]{Bieri}.

\begin{lem}
\label{fpnfg}
If a pro-$\mathcal{C}$ group $G$ is of type $\FP_n$ over $\mathbb{Z}_{\hat{\mathcal{C}}}$ then $H^{\mathbb{Z}_{\hat{\mathcal{C}}}}_m(G, \mathbb{Z}_{\hat{\mathcal{C}}})$ is a finitely generated profinite abelian group for $0 \leq m \leq n$.
\end{lem}
\begin{proof}
Take a projective resolution $P_\ast$ of $\mathbb{Z}_{\hat{\mathcal{C}}}$ as a $\mathbb{Z}_{\hat{\mathcal{C}}} \llbracket G \rrbracket $-module with trivial $G$-action, with $P_0, \ldots, P_n$ finitely generated. Then $H^{\mathbb{Z}_{\hat{\mathcal{C}}}}_\ast(G,\mathbb{Z}_{\hat{\mathcal{C}}})$ is the homology of the complex $\mathbb{Z}_{\hat{\mathcal{C}}} \hat{\otimes}_{\mathbb{Z}_{\hat{\mathcal{C}}} \llbracket G \rrbracket } P_\ast$, for which $$\mathbb{Z}_{\hat{\mathcal{C}}} \hat{\otimes}_{\mathbb{Z}_{\hat{\mathcal{C}}} \llbracket G \rrbracket } P_0, \ldots, \mathbb{Z}_{\hat{\mathcal{C}}} \hat{\otimes}_{\mathbb{Z}_{\hat{\mathcal{C}}} \llbracket G \rrbracket } P_n$$ are finitely generated $\mathbb{Z}_{\hat{\mathcal{C}}}$-modules. Now $\mathbb{Z}_{\hat{\mathcal{C}}}$ is procyclic, hence a principal ideal domain, which implies by standard arguments that $\mathbb{Z}_{\hat{\mathcal{C}}}$ is noetherian in the sense that submodules of finitely generated $\hat{\mathbb{Z}}$-modules are finitely generated, and the result follows: finitely generated pro-$\mathcal{C}$ abelian groups are exactly the finitely generated pro-$\mathcal{C}$ $\mathbb{Z}_{\hat{\mathcal{C}}}$-modules.
\end{proof}

\begin{prop}
\begin{enumerate}[(i)]
\item $A_n$ is of type $\FP_n$ over $\mathbb{Z}_{\hat{\mathcal{C}}}$.
\item $A_n$ is not of type $\FP_{n+1}$ over $\mathbb{Z}_{\hat{\mathcal{C}}}$.
\end{enumerate}
\end{prop}
\begin{proof}
\begin{enumerate}[(i)]
\item $n=0$ is trivial. Next, we observe that $A_n$ can be thought of as the pro-$\mathcal{C}$ $\HNN$-extension of $A_{n-1} \rtimes \langle x_n \rangle$ with associated subgroup $A_{n-1}$ and stable letter $y_n$, since the universal properties are the same in this case. It is clear that this extension is proper.

We can now use induction: assume $A_{n-1}$ is of type $\FP_{n-1}$ over $\mathbb{Z}_{\hat{\mathcal{C}}}$ (which we already have for $n=0$). Then $A_{n-1} \rtimes \langle x_n \rangle$ is of type $\FP_\infty$, so the result follows from Proposition \ref{hnn'}.

\item By Lemma \ref{fpnfg}, it is enough to show that, for each $n$, $H^{\mathbb{Z}_{\hat{\mathcal{C}}}}_{n+1}(A_n,\mathbb{Z}_{\hat{\mathcal{C}}})$ is not finitely generated. We prove this by induction once more. Exactly as in \cite[Lemma 6.8.6]{R-Z}, $$H^{\mathbb{Z}_{\hat{\mathcal{C}}}}_1(A_0,\mathbb{Z}_{\hat{\mathcal{C}}}) = F_{\mathbb{Z}_{\hat{\mathcal{C}}}}/\overline{[F_{\mathbb{Z}_{\hat{\mathcal{C}}}},F_{\mathbb{Z}_{\hat{\mathcal{C}}}}]},$$ i.e. the pro-$\mathcal{C}$ free abelian group on the set $\mathbb{Z}_{\hat{\mathcal{C}}}$, not finitely generated. As before, $A_n$ is the pro-$\mathcal{C}$ $\HNN$-extension of $A_{n-1} \rtimes \langle x_n \rangle$ with associated subgroup $A_{n-1}$ and stable letter $y_n$, and we get the Mayer-Vietoris sequence

\begin{align*}
\cdots &\rightarrow H^{\mathbb{Z}_{\hat{\mathcal{C}}}}_{n+1}(A_{n-1} \rtimes \langle x_n \rangle,\mathbb{Z}_{\hat{\mathcal{C}}}) \rightarrow H^{\mathbb{Z}_{\hat{\mathcal{C}}}}_{n+1}(A_n,\mathbb{Z}_{\hat{\mathcal{C}}}) \rightarrow H^{\mathbb{Z}_{\hat{\mathcal{C}}}}_n(A_{n-1},\mathbb{Z}_{\hat{\mathcal{C}}}) \\
&\rightarrow H^{\mathbb{Z}_{\hat{\mathcal{C}}}}_n(A_{n-1} \rtimes \langle x_n \rangle,\mathbb{Z}_{\hat{\mathcal{C}}}) \rightarrow \cdots
\end{align*}

By Lemma \ref{fpnfg} $H^{\mathbb{Z}_{\hat{\mathcal{C}}}}_{n+1}(A_{n-1} \rtimes \langle x_n \rangle,\mathbb{Z}_{\hat{\mathcal{C}}})$ and $H^{\mathbb{Z}_{\hat{\mathcal{C}}}}_n(A_{n-1} \rtimes \langle x_n \rangle,\mathbb{Z}_{\hat{\mathcal{C}}})$ are finitely generated; by hypothesis $H^{\mathbb{Z}_{\hat{\mathcal{C}}}}_n(A_{n-1},\mathbb{Z}_{\hat{\mathcal{C}}})$ is not finitely generated. Hence $H^{\mathbb{Z}_{\hat{\mathcal{C}}}}_{n+1}(A_n,\mathbb{Z}_{\hat{\mathcal{C}}})$ is not finitely generated, as required.
\end{enumerate}
\end{proof}

\section{An alternative finiteness condition}
\label{altfin}

According to \cite[Open Question 6.12.1]{R-Z}, Kropholler has posed the question: ``Let $G$ be a soluble pro-$p$ group such that $H^n(G,\mathbb{Z}/p\mathbb{Z})$ is finite for every $n$. Is $G$ poly[-pro]cyclic?''. Now, we know by \cite[Corollary 4.2.5]{S-W} that requiring $H^n(G,\mathbb{Z}/p\mathbb{Z})$ to be finite for every $n$ is equivalent to requiring that $G$ be of type $p$-$\FP_\infty$, and by \cite[Proposition 4.2.3]{S-W} equivalent to requiring that $H^n(G,A)$ is finite for every $n$ and every finite $\mathbb{Z}_p \llbracket G \rrbracket $-module $A$. Also, by \cite[Proposition 8.2.2]{Wilson}, $G$ is poly-procyclic if and only if it has finite rank. So there are two possible profinite analogues of this question, either of which, if the answer were yes, would imply \cite[Open Question 6.12.1]{R-Z}.

\begin{question}
\label{oq1}
Let $G$ be a soluble profinite group such that $H^n(G,A)$ is finite for every $n$ and every finite $\hat{\mathbb{Z}} \llbracket G \rrbracket $-module $A$. Is $G$ of finite rank?
\end{question}

\begin{question}
\label{oq2}
Let $G$ be a soluble profinite group of type $\FP_\infty$ over $\hat{\mathbb{Z}}$. Is $G$ of finite rank?
\end{question}

We will show that the answer to the first of these questions is no. Question \ref{oq2} remains open.

In this section, all modules will be left modules.

By analogy to the pro-$p$ case, we define profinite $G$ to be \emph{of type $\FP'_n$} (over $\hat{\mathbb{Z}}$) if, for all finite $\hat{\mathbb{Z}} \llbracket G \rrbracket $-modules $A$, $m \leq n$, $H_{\hat{\mathbb{Z}}}^m(G,A)$ is finite. This definition extends in the obvious way to profinite modules over any profinite ring. Clearly, by the Lyndon-Hochschild-Serre spectral sequence \cite[Theorem 7.2.4]{R-Z}, being of type $\FP'_n$ is closed under extensions. In the same way as \cite[Proposition 4.2.2]{S-W}, $\FP_n \Rightarrow \FP'_n$ for all $n \leq \infty$; in this section we will see that the converse is not true.

\begin{lem}
\label{sylow}
If $G$ is pronilpotent of type $\FP_1$, the minimal number of generators of its $p$-Sylow subgroups is bounded above.
\end{lem}
\begin{proof}
For $G$ pronilpotent, by \cite[Proposition 2.3.8]{R-Z}, $G$ is the direct product of its (unique for each $p$) $p$-Sylow subgroups. If $G$ is finitely generated, pick a set of generators for $G$; then their images in each $p$-Sylow subgroup under the canonical projection map generate that subgroup. Hence the minimal number of generators of the $p$-Sylow subgroups of $G$ is bounded above. We know $G$ is \textit{a fortiori} prosoluble, so by Proposition \ref{fingen} and Remark \ref{prosoluble}(a) $G$ is of type $\FP_1$ if and only if it is finitely generated, and the result follows.
\end{proof}

\begin{lem}
\label{coprime}
Suppose $A$ is a finite $G$-module whose order is coprime to that of $G$. Then $H_{\hat{\mathbb{Z}}}^n(G,A)$ is $0$ for all $n>0$.
\end{lem}
\begin{proof}
By \cite[Corollary 7.3.3]{R-Z}, $cd_p(G)=0$ for $p \nmid |G|$. In particular, $$H_{\hat{\mathbb{Z}}}^n(G,A)_p=0 \text{ for all } p \mid |A|, n>0.$$ On the other hand, by \cite[Proposition 7.1.4]{R-Z}, $$H_{\hat{\mathbb{Z}}}^n(G,A) = \bigoplus_{p \mid |A|} H_{\hat{\mathbb{Z}}}^n(G,A_p) = \bigoplus_{p \mid |A|} H_{\hat{\mathbb{Z}}}^n(G,A)_p = 0.$$
\end{proof}

\begin{prop}
\label{fp'n}
Let $G$ be pronilpotent. Then $G$ is of type $\FP'_n$ if and only if every $p$-Sylow subgroup is of type $\FP'_n$.
\end{prop}
\begin{proof}
Suppose every $p$-Sylow subgroup is of type $\FP'_n$. Suppose $A$ is a finite $\hat{\mathbb{Z}} \llbracket G \rrbracket $-module. Now $A$ is finite, so only finitely many primes divide the order of $A$. Suppose $p_1, \ldots, p_m$ are the primes for which $p_i \mid |A|$, and write $\pi$ for the set of primes without $p_1, \ldots, p_m$. Write $G$ again as the direct product of its $p$-Sylow subgroups, $G= \prod_p S_p$. By the Lyndon-Hochschild-Serre spectral sequence (\cite[Theorem 7.2.4]{R-Z}) $\prod_{i=1}^m S_{p_i}$ is of type $\FP'_n$. Thus, applying the spectral sequence again, $H_{\hat{\mathbb{Z}}}^{r+s}(G,A)$ is a sequence of extensions of the groups $H_{\hat{\mathbb{Z}}}^r(\prod_{i=1}^m S_{p_i}, H_{\hat{\mathbb{Z}}}^s(\prod_{p \in \pi} S_p, A))$, which by Lemma \ref{coprime} collapses to give $H_{\hat{\mathbb{Z}}}^{r+s}(\prod_{i=1}^m S_{p_i}, A^{\prod_{p \in \pi} S_p})$, finite.

Conversely, if some $S_p$ is not of type $\FP'_n$, there is some $S_p$-module $A$ and some $k \leq n$ such that $H_{\hat{\mathbb{Z}}}^k(S_p,A)$ is infinite. All groups are of type $\FP_0$ and hence of type $\FP'_0$, so we have $k>0$. Then by Lemma \ref{coprime} we have that $H_{\hat{\mathbb{Z}}}^k(S_p,A) = \bigoplus_{p' \mid |A|} H_{\hat{\mathbb{Z}}}^k(S_p,A_{p'}) = H_{\hat{\mathbb{Z}}}^k(S_p,A_p)$ is infinite, and so we may assume $A = A_p$. Then we can make $A$ a $G$-module by having every $S_{p'}$, $p' \neq p$, act trivially on $A$, and the spectral sequence together with Lemma \ref{coprime} gives that $H_{\hat{\mathbb{Z}}}^k(G,A) = H_{\hat{\mathbb{Z}}}^k(S_p,A^{\prod_{p' \neq p} S_{p'}}) = H_{\hat{\mathbb{Z}}}^m(S_p,A)$, which is infinite, and hence $G$ is not of type $\FP'_n$.
\end{proof}

Finally, as promised, we will answer Question \ref{oq1} by constructing a soluble (in fact torsion-free abelian) profinite group of type $\FP'_\infty$ which is not finitely generated, and hence not of type $\FP_1$ by Proposition \ref{fingen} and Remarks \ref{prosoluble}(a), and not of finite rank.

\begin{example}
\label{fp'notfg}
Write $p_n$ for the $n$th prime, and consider the abelian profinite group $G = \prod_n (\prod_{i=1}^n \mathbb{Z}_{p_n})$. By Lemma \ref{sylow}, $G$ is not of type $\FP_1$. By Example \ref{procyclic}, and the Lyndon-Hochschild-Serre spectral sequence, the $p_n$-Sylow subgroup $\prod_{i=1}^n \mathbb{Z}_{p_n}$ of $G$ is of type $\FP_\infty$ for each $n$, and hence of type $\FP'_\infty$. So by Proposition \ref{fp'n}, $G$ is of type $\FP'_\infty$.
\end{example}

\end{document}